\documentclass
[
    fontsize = 11 pt,
    american,
    captions = tableheading,
    numbers = noenddot,
    abstracton,
    paper = letter,
]
{scrartcl}

\usepackage[T1]{fontenc}
\usepackage[utf8]{inputenc}
\usepackage
[
    babel = true, 
]
{microtype}           
\usepackage{csquotes} 

\usepackage[dvipsnames]{xcolor} 

\definecolor{stroke1}{HTML}{2574A9} 

\usepackage{amsmath}
\usepackage{amssymb}
\usepackage{amsthm}
\usepackage{thmtools}
\usepackage{mathtools}
\usepackage{thm-restate}
\usepackage{dsfont}        
\usepackage{braceMnSymbol} 

\usepackage
[
    ttscale = 0.85, 
]
{libertine} 
\usepackage
[
    libertine,    
    slantedGreek, 
    vvarbb,       
    libaltvw,     
]
{newtxmath} 
\usepackage{url} 
\usepackage{bm}  

\usepackage{graphicx} 
\usepackage
[
    subrefformat = simple, 
    labelformat = simple,  
]
{subcaption}         

\usepackage{array}     
\usepackage{booktabs}  
\usepackage{longtable} 
\usepackage{pdflscape} 

\usepackage[shortlabels]{enumitem} 

\usepackage
[
    ruled,         
    vlined,        
    linesnumbered, 
]
{algorithm2e} 

\usepackage{xspace}   
\usepackage
[
    shortcuts, 
]
{extdash}             
\usepackage{setspace} 

\usepackage{xparse}    
\usepackage{footnote}  
\usepackage{afterpage} 
\usepackage
[
    textsize = scriptsize, 
]
{todonotes}            

\usepackage
[
    sortcites,              
    style = alphabetic,     
    defernumbers,           
    safeinputenc,           
    backref = true,         
    backrefstyle = three,   
    hyperref = true,        
    maxbibnames = 99,       
    mincitenames = 99,       
    maxcitenames = 99,       
    minalphanames = 3,       
]
{biblatex} 

\DefineBibliographyStrings{english}%
{%
    backrefpage  = {\lowercase{s}ee page}, 
    backrefpages = {\lowercase{s}ee pages} 
}

\usepackage{multirow}
\usepackage{caption}
\usepackage
[
    bookmarks = true,                 
    bookmarksopen = false,            
    bookmarksnumbered = true,         
    pdfstartpage = 1,                 
    colorlinks = true,                
    allcolors = stroke1,              
]
{hyperref} 

\usepackage
[
    noabbrev,   
    nameinlink, 
]
{cleveref} 

\date{}

\makeatletter
    \def\IfEmptyTF#1%
    {%
        \if\relax\detokenize{#1}\relax%
            \expandafter\@firstoftwo%
        \else%
            \expandafter\@secondoftwo%
        \fi%
    }
\makeatother

\NewDocumentCommand{\mathOrText}{m}
{%
    \ensuremath{#1}\xspace%
}

\let\originalleft\left
\let\originalright\right
\renewcommand{\left}{\mathopen{}\mathclose\bgroup\originalleft}
\renewcommand{\right}{\aftergroup\egroup\originalright}

\makeatletter
    \DeclareRobustCommand{\bfseries}%
    {%
        \not@math@alphabet\bfseries\mathbf%
        \fontseries\bfdefault\selectfont%
        \boldmath%
    }

\xspaceaddexceptions{]\}}

\allowdisplaybreaks

\crefname{ineq}{inequality}{inequalities}
\creflabelformat{ineq}{#2{\upshape(#1)}#3}

\crefname{term}{term}{terms}
\creflabelformat{term}{#2{\upshape(#1)}#3}

\crefname{cond}{condition}{conditions}
\creflabelformat{cond}{#2{\upshape(#1)}#3}

\crefname{assume}{assumption}{assumptions}
\creflabelformat{assume}{#2{\upshape(#1)}#3}

\deffootnote[1.2 em]{1.2 em}{0 em}{\makebox[1.2 em][l]{\textbf{\thefootnotemark}}}

\let\oldfootnote\footnote

\newlength{\spaceBeforeFootnote} 
\newlength{\spaceAfterFootnote}  

\RenewDocumentCommand{\footnote}{o o o m}%
{%
    \IfNoValueTF{#1}%
    {%
        \oldfootnote{#4}%
    }%
    {%
        \setlength{\spaceBeforeFootnote}{\IfEmptyTF{#1}{0}{#1} em}%
        \IfNoValueTF{#2}%
        {%
            \hspace*{\spaceBeforeFootnote}\oldfootnote{#4}%
        }%
        {%
            \setlength{\spaceAfterFootnote}{\IfEmptyTF{#2}{0}{#2} em}%
            \hspace*{\spaceBeforeFootnote}\IfNoValueTF{#3}{\oldfootnote{#4}}{\oldfootnote[#3]{#4}}\hspace*{\spaceAfterFootnote}%
        }%
    }%
}

\makesavenoteenv{figure}
\makesavenoteenv{table}
\makesavenoteenv{tabular}

\declaretheoremstyle
[
   	spaceabove = \topsep,
   	spacebelow = \topsep,
   	headfont = \bfseries,
   	headformat = \textcolor{stroke1}{$\blacktriangleright$} \NAME~\NUMBER \NOTE,
   	notefont = \bfseries,
   	notebraces = {(}{)},
   	bodyfont = \normalfont,
   	postheadspace = 0.5 em,
   	qed = \textcolor{stroke1}{\bfseries$\blacktriangleleft$},
]
{myTheoremStyle}

\declaretheorem
[
   	style = myTheoremStyle,
   	name = Lemma,
    sharenumber = conjecture,
]
{lemma}
\declaretheorem
[
   	style = myTheoremStyle,
   	name = Corollary,
    sharenumber = conjecture,
]
{corollary}
\declaretheorem
[
   	style = myTheoremStyle,
   	name = Theorem,
    sharenumber = conjecture,
]
{theorem}
\declaretheorem
[
   	style = myTheoremStyle,
   	name = Definition,
    sharenumber = conjecture,
]
{definition}

\NewDocumentCommand{\functionTemplate}{m m m m o}%
{%
    \IfNoValueTF{#5}%
    {%
        \mathOrText{#1\left#2{#4}\right#3}%
    }%
    {%
        \mathOrText{#1#5#2{#4}#5#3}%
    }%
}

\newcommand*{\leftBracketType}{(}
\newcommand*{\rightBracketType}{)}

\NewDocumentCommand{\createFunction}{m m o o}%
{%
    \renewcommand*{\leftBracketType}{\IfNoValueTF{#3}{(}{#3}}%
    \renewcommand*{\rightBracketType}{\IfNoValueTF{#4}{)}{#4}}%
    \NewDocumentCommand{#1}{o o}%
    {%
        \IfNoValueTF{##1}%
        {%
            \mathOrText{#2}%
        }%
        {%
            \functionTemplate{#2}{\leftBracketType}{\rightBracketType}{##1}[##2]%
        }%
    }%
}

\DeclareDocumentCommand{\probabilisticFunctionTemplate}{m m O{} o}
{%
    \functionTemplate{#1}%
    {\lbrack}%
    {\rbrack}%
    {#2\IfEmptyTF{#3}{}{\ \IfNoValueTF{#4}{\left}{#4}\vert\ \vphantom{#2}#3\IfNoValueTF{#4}{\right.}{}}}%
    [#4]%
}

\newcommand*{\N}{\mathOrText{\mathds{N}}}
\newcommand*{\Z}{\mathOrText{\mathds{Z}}}

\newcommand*{\R}{\mathOrText{\mathds{R}}}

\newcommand*{\indicatorFunctionSymbol}{\mathds{1}}

\RenewDocumentCommand{\Pr}{m O{} o}%
{%
    \probabilisticFunctionTemplate{\mathrm{Pr}}{#1}[#2][#3]%
}

\NewDocumentCommand{\E}{m O{} o}%
{%
    \probabilisticFunctionTemplate{\mathds{E}}{#1}[#2][#3]%
}

\NewDocumentCommand{\Var}{m O{} o}%
{%
    \probabilisticFunctionTemplate{\mathrm{Var}}{#1}[#2][#3]%
}

\DeclareDocumentCommand{\bigO}{m o}%
{%
    \functionTemplate{\mathrm{O}}{(}{)}{#1}[#2]%
}

\DeclareDocumentCommand{\smallO}{m o}%
{%
    \functionTemplate{\mathrm{o}}{(}{)}{#1}[#2]%
}

\DeclareDocumentCommand{\bigTheta}{m o}%
{%
    \functionTemplate{\upTheta}{(}{)}{#1}[#2]%
}

\DeclareDocumentCommand{\bigOmega}{m o}%
{%
    \functionTemplate{\upOmega}{(}{)}{#1}[#2]%
}

\DeclareDocumentCommand{\smallOmega}{m o}%
{%
    \functionTemplate{\upomega}{(}{)}{#1}[#2]%
}

\DeclareDocumentCommand{\eulerE}{o}%
{%
    \mathOrText{\mathrm{e}\IfNoValueTF{#1}{}{^{#1}}}%
}

\DeclareDocumentCommand{\poly}{m o}%
{%
    \functionTemplate{\mathrm{poly}}{(}{)}{#1}[#2]%
}

\createFunction{\id}{\mathrm{id}}

\NewDocumentCommand{\ind}{m o o}%
{%
    \IfNoValueTF{#2}%
    {%
        \mathOrText{\indicatorFunctionSymbol_{#1}}%
    }%
    {%
        \functionTemplate{\indicatorFunctionSymbol_{#1}}{(}{)}{#2}[#3]%
    }%
}

\DeclareDocumentCommand{\dom}{m o}%
{%
    \functionTemplate{\mathrm{dom}}{(}{)}{#1}[#2]%
}

\DeclareDocumentCommand{\rng}{m o}%
{%
    \functionTemplate{\mathrm{rng}}{(}{)}{#1}[#2]%
}

\DeclareDocumentCommand{\d}{o}%
{%
    \mathrm{d}\IfNoValueTF{#1}{}{^{#1}}%
}

\DeclareDocumentCommand{\set}{m m o}%
{%
    \mathOrText{\IfNoValueTF{#3}{\left}{#3}\{#1\ \IfNoValueTF{#3}{\left}{#3}\vert\ \vphantom{#1}#2\IfNoValueTF{#3}{\right.}{}\IfNoValueTF{#3}{\right}{#3}\}}%
}

\newcommand*{\indicator}[1]{\mathOrText{\mathds{1}_{#1}}}

\DeclareDocumentCommand{\randomProcess}{m o}
{%
    \mathOrText{X^{(#1)}\IfNoValueTF{#2}{}{_{#2}}}%
}

\DeclareDocumentCommand{\transformedProcess}{o}
{%
    \mathOrText{Y\IfNoValueTF{#1}{}{_{#1}}}%
}

\DeclareDocumentCommand{\filtration}{o}
{%
    \mathOrText{\mathcal{F}\IfNoValueTF{#1}{}{_{#1}}}%
}

\newcommand*{\timePoint}{\mathOrText{t}}

\newcommand*{\defeq}{\mathOrText{\coloneqq}}

\newcommand*{\hyperbolicExponent}{\mathOrText{\gamma}}
\newcommand*{\numberOfVertices}{\mathOrText{n}}
\newcommand*{\numberOfVerticesShifted}{\mathOrText{n'}}
\newcommand*{\infectionRate}{\mathOrText{\lambda}}
\newcommand*{\infectionConstant}{\mathOrText{c}}

\newcommand*{\infectionConstantMinusOne}{\mathOrText{c'}}
\newcommand*{\deimmunizationRate}{\mathOrText{\varrho}}
\newcommand*{\contactProcess}{\mathOrText{C}}
\newcommand*{\contactProcessProjection}{\mathOrText{C'}}
\newcommand*{\infectedDiscrete}[1]{\mathOrText{I_{\timeContinuous{#1}}}}

\newcommand*{\susceptibleDiscrete}[1]{\mathOrText{S_{\timeContinuous{#1}}}}
\newcommand*{\recoveredDiscrete}[1]{\mathOrText{R_{\timeContinuous{#1}}}}
\newcommand*{\susceptibleDiscreteShifted}[1]{\mathOrText{P_{\timeContinuous{#1}}}}

\newcommand*{\infectedEquilibrium}{\mathOrText{I^*}}

\newcommand*{\timeDiscrete}{\mathOrText{t}}
\newcommand*{\timeContinuous}[1]{\mathOrText{\tau_{#1}}}

\newcommand*{\eventInfection}[1]{\mathOrText{E_{si,#1}}}
\newcommand*{\eventInfectionOutside}[1]{\mathOrText{E_{o,#1}}}
\newcommand*{\eventRecover}[1]{\mathOrText{E_{ir,#1}}}
\newcommand*{\eventSusceptible}[1]{\mathOrText{E_{rs,#1}}}

\newcommand*{\rateInfection}[1]{\mathOrText{r_{si,#1}}}
\newcommand*{\rateInfectionOutside}[1]{\mathOrText{r_{o,#1}}}
\newcommand*{\rateRecover}[1]{\mathOrText{r_{ir,#1}}}
\newcommand*{\rateSusceptible}[1]{\mathOrText{r_{rs,#1}}}

\newcommand*{\rateTotal}[1]{\mathOrText{r'_{#1}}}
\newcommand*{\rateExtra}[1]{\mathOrText{r_{c,#1}}}
\newcommand*{\degree}{\mathOrText{d}}
\newcommand*{\degreeGap}{\mathOrText{\varepsilon_d}}
\newcommand*{\spectralGap}{\mathOrText{\delta}}
\newcommand*{\rateTotalOutside}[1]{\mathOrText{r_{#1}}}
\newcommand*{\probabilityInfection}[1]{\mathOrText{p_{si,#1}}}
\newcommand*{\probabilityInfectionOutside}[1]{\mathOrText{p_{o,#1}}}
\newcommand*{\probabilityRecover}[1]{\mathOrText{p_{ir,#1}}}
\newcommand*{\probabilitySusceptible}[1]{\mathOrText{p_{rs,#1}}}

\newcommand*{\driftRecover}[1]{\mathOrText{F_{ir,#1}}}
\newcommand*{\driftSusceptible}[1]{\mathOrText{F_{rs,#1}}}
\newcommand*{\driftInfection}[1]{\mathOrText{F_{si,#1}}}

\DeclareDocumentCommand{\lyapunovHelper}{o}
{%
    \mathOrText{f\IfNoValueTF{#1}{}{(#1)}}%
}

\DeclareDocumentCommand{\lyapunovFunction}{m o}
{%
    \mathOrText{F\IfNoValueTF{#2}{}{(#2)}}%
}

\newcommand*{\lyapunovConstant}{\mathOrText{\alpha}}
\newcommand*{\lyapunovParameter}{\mathOrText{\beta}}
\newcommand*{\potentialSIRSClique}[2]{\mathOrText{F_{#2}}}
\newcommand*{\driftSIRSClique}[2]{\mathOrText{D_{#2}}}

\DeclareDocumentCommand{\potentialFunctionIMinusR}{o}
{%
    \mathOrText{H\IfNoValueTF{#1}{}{(#1)}}%
}
\newcommand*{\potentialIMinusR}[1]{\mathOrText{H_{#1}}}
\newcommand*{\potentialIMinusRAlone}{\mathOrText{H}}

\newcommand*{\filtrationDiscrete}[1]{\mathOrText{\filtration_{\timeContinuous{#1}}}}
\newcommand*{\filtrationContinuous}[1]{\mathOrText{\filtration_{#1}}}

\newcommand*{\infectedContinuous}[1]{\mathOrText{I_{#1}}}
\newcommand*{\susceptibleContinuous}[1]{\mathOrText{S_{#1}}}
\newcommand*{\recoveredContinuous}[1]{\mathOrText{R_{#1}}}
\newcommand*{\infectedSet}[1]{\mathOrText{I_{#1}'}}
\newcommand*{\susceptibleSet}[1]{\mathOrText{S_{#1}'}}
\newcommand*{\recoveredSet}[1]{\mathOrText{R_{#1}'}}

\newcommand*{\erdosGraph}{Erdős--Rényi graph\xspace}
\newcommand*{\erdosGraphs}{Erdős--Rényi graphs\xspace}

\DeclareDocumentCommand{\vol}{m}
{%
    \mathOrText{\text{vol}(#1)}%
}
\DeclareDocumentCommand{\edges}{m m}
{%
    \mathOrText{E(#1,#2)}%
}

\title{Analysis of the survival time of the SIRS~process via expansion}

\author{
    Tobias Friedrich$^{*}$ \and Andreas Göbel$^{*}$ \and Nicolas Klodt$^{*}$ \and Martin~S. Krejca$^{\dagger}$ \and Marcus Pappik$^{*}$
}

\publishers
 {
  	\footnotesize $^{*}$ Hasso Plattner Institute, University of Potsdam, Potsdam, Germany\\
  	\{tobias.friedrich, andreas.goebel, nicolas.klodt, marcus.pappik\}@hpi.de \\[1 ex]

  	$^{\dagger}$ LIX, CNRS, Ecole Polytechnique, Institut Polytechnique de Paris, Palaiseau, France \\
  	martin.krejca@polytechnique.edu
 }

\begin{document}

\maketitle

\vspace*{-2.5\baselineskip}
\begin{abstract}
We study the SIRS process, a continuous-time Markov chain modeling the spread of infections on graphs.
In this process, vertices are either susceptible, infected, or recovered.
Each infected vertex becomes recovered at rate 1 and infects each of its susceptible neighbors independently at rate~$\infectionRate$, and each recovered vertex becomes susceptible at a rate~$\deimmunizationRate$, which we assume to be independent of the graph size.
A central quantity of the SIRS process is the time until no vertex is infected, known as the \emph{survival time}.
The survival time of the SIRS process is studied extensively in a variety of contexts.
Surprisingly though, to the best of our knowledge, no rigorous theoretical results exist so far.
This is even more surprising given that for the related SIS process, mathematical analysis began in the 70s and continues to this day.

We address this imbalance by conducting the first theoretical analyses of the SIRS process on various graph classes via their expansion properties.
Our analyses assume that the graphs start with at least one infected vertex and no recovered vertices.
Our first result considers stars, which have poor expansion.
We prove that the expected survival time of the SIRS process on stars is at most polynomial in the graph size for \emph{any} value of $\infectionRate$.
This behavior is fundamentally different from the SIS process, where the expected survival time is exponential already for small infection rates.
Due to this property, for the SIS process, stars constitute an important sub-structure for proving an expected exponential survival time of more complicated graphs.
For the SIRS process, this argument is not sufficient.

Our main result is an exponential lower bound of the expected survival time of the SIRS process on expander graphs.
Specifically, we show that on expander graphs~$G$ with~$n$ vertices, degree close to~$\degree$, and sufficiently small spectral expansion, the SIRS process has expected survival time at least exponential in~$n$ when $\infectionRate \geq c/\degree$ for a constant $c > 1$.
This result is complemented by established results for the SIS process, which imply that the expected survival time of the SIRS process is at most logarithmic in~$n$ when $\infectionRate \leq c/\degree$ for a constant $c < 1$.
Combined, our result shows an almost-tight threshold behavior of the expected survival time of the SIRS process on expander graphs.
Additionally, our result holds even if~$G$ is a subgraph.
This allows, for the SIRS process, the use of expanders as sub-structures for lower bounds, similar to stars in the SIS process.
Notably, our result implies an almost-tight threshold for \erdosGraphs and a regime of exponential survival time for hyperbolic random graphs, one of the most popular graph models, as it incorporates many properties found in real-world networks.
The proof of our main result draws inspiration from Lyapunov functions used in mean-field theory to devise a two-dimensional potential function and applying a negative-drift theorem to show that the expected survival time is exponential.
\end{abstract}

\newpage

\section{Introduction}
\label{sec:introduction}
In the domain of modeling infectious diseases, a vast body of literature studying various stochastic processes on graphs exists (see, for example, the extensive survey by~\textcite{Pastor-SatorrasCVMV15Survey}). In this article, we focus on the SIRS process, a continuous-time Markov chain where each vertex is either susceptible, infected, or recovered. Each infected vertex becomes recovered at rate~$1$ and infects each of its susceptible neighbors independently at an \emph{infection rate}~$\infectionRate$, while each recovered vertex becomes susceptible at a \emph{deimmunization rate}~$\deimmunizationRate$.

A question central to understanding the SIRS process is \emph{how long} it takes until no vertex in the graph is infected, known as the \emph{survival time}\footnote{Sometimes also referred to as the \emph{extinction} time.} of the process. Due to relevance of the SIRS process, its survival time has been studied extensively. This includes empirical results~\cite{Wang_2017,PhysRevLett.86.2909,ferreira2016collective}, mean-field approaches~\cite{Bancal10}, and results that consider deterministic variants of the process~\cite{saif2019epidemic}. However, surprisingly, to the best of our knowledge, no rigorous, theoretical results exist for the SIRS process in the literature.

This lack of theoretical results for the SIRS process stands in stark contrast to the plethora of theoretical results for a similar but slightly simpler process, known as the SIS process.
In the SIS process, each vertex is either susceptible or infected.
Each infected vertex becomes susceptible at rate~$1$ and infects each of its neighbors independently at an \emph{infection rate}~$\infectionRate$.
Thus, with a grain of salt, the SIS process can be viewed as a special case of the SIRS process in which recovered vertices turn immediately susceptible (that is, the deimmunization rate~$\deimmunizationRate$ is infinite).
The survival time of the SIS process is well understood on a variety of graphs.
Early results on the SIS process consider its survival time on $\Z^d$~\cite{Harris74} and on infinite $d$-regular trees~\cite{Pemantle1992TheCP,10.2307/2959578,Stacey96}, while recent breakthroughs characterize the survival time on Galton--Watson trees~\cite{huang2020contact,bhamidi2021survival,NamNS22SISinfinite}.
On finite structures, the results of \textcite{NamNS22SISinfinite} consider \erdosGraphs, while the SIS process has also been studied on scale-free graphs\footnote{Generated by the preferential-attachment model~\cite{barabasi1999emergence}.}~\cite{berger2005spread,BorgsCGS10Antidote}.
These results rely on the survival time on simple subgraphs, such as stars.
Further, \textcite{ganesh2005effect} connect the survival time to the spectral radius and the isoperimetric constant of the host graph, which immediately translates to a variety of simple~graphs.

We note that, for the same graph, the survival time---which is a random variable---of a SIS process is an upper bound for the survival time of a SIRS process when starting with identical configurations, as the two processes can be coupled such that an infected vertex in the latter is also always infected in the former.
This allows to carry over \emph{some} results from the SIS to the SIRS process.
However, our knowledge about the SIRS process remains in a very unsatisfactory state for multiple reasons.
First, we only have upper bounds on the survival time for the SIRS process, which begs the question for how tight they are.
And second, far more importantly, the survival time in the SIS process for a graph~$G$ is a lower bound for any graph~$H$ containing~$G$ as a subgraph, as adding more vertices does not reduce the number of infected vertices at any point in time.
In contrast, it is not known whether the SIRS process also has this property.
Adding more vertices to a graph in the SIRS process can lead to some vertices being earlier infected and thus potentially earlier recovered, which in turn can block an infection that would have occurred otherwise.
Thus, it is not straightforward to generalize results for the SIRS process to supergraphs.

\paragraph{Our contribution.}
We conduct the first rigorous, theoretical study of the expected survival time of the SIRS process on a large variety of graph classes, most prominently expanders.
In all of our results, we assume that the deimmunization rate is independent of the graph size and that the process starts with at least one infected vertex and no recovered vertices.
Our results showcase the similarities and the differences between the SIS and the SIRS process, highlighting the impact of the state \emph{recovered}.
Furthermore, for our lower bounds, we prove that our results carry over to supergraphs of the graphs we analyze.
This makes our results applicable to a great number of different graph classes.

More specifically, in \Cref{sec:stars}, we show that the expected survival time of the SIRS process on stars is polynomial,\footnote[-0.1]{In the number of vertices.} regardless of the infection rate (\Cref{lem:starSurvival}).
This strongly contrasts the SIS process, where the survival time is superpolynomial for already very small infection rates.
This shows that recovered vertices can have a huge impact on the survival time.
The reason for this drastic difference in the expected survival time between both processes is that the star is only connected through a single, central vertex.
Thus, if the center is recovered, the infection only survives if not all leaves become recovered during this time interval. The latter event does not have sufficiently high probability of occurring for the infection to survive superpolynomially long.

In \Cref{sec:SIRS_clique}, we complement these findings by proving that the expected survival time of the SIRS process on expanders is at least exponential if the infection rate is greater than the inverse of the expander's average degree (\Cref{thm:cliqueSIRS}).
This result is very similar for the SIS process~\cite{ganesh2005effect}.
In contrast to stars, expanders have many edges between arbitrary subsets of vertices.
Thus, if the number of infected vertices is sufficiently high, there exist enough edges between all susceptible and all infected vertices, regardless of the number of (remaining) recovered vertices.
These edges give the process a high probability to not decrease the number of infected vertices, which leads to the overall long expected survival time.

Since we prove our result for expanders to carry over to supergraphs, this result implies respective expected survival times for other well known graph classes, such as \erdosGraphs (\Cref{cor:ER_graphs}) and hyperbolic random graphs (\Cref{cor:SIRS_on_HRGs}), which we discuss in \Cref{sec:graph_classes}.
Combined, our results emphasize that while the SIRS and SIS process behave very differently on some of their subgraphs (namely stars), they have similar behavior if the graph is sufficiently connected.
In the following, we discuss our results in more detail.

\subsection{Expected survival time on stars}

For stars, we prove the following upper bound on the expected survival time of the SIRS process.

\begin{restatable*}{theorem}{StarSurvival}
    \label{lem:starSurvival}
    Let $G$ be a star with $\numberOfVertices \in \N_{>0}$ leaves, and let \contactProcess be a SIRS process on $G$ with infection rate \infectionRate and with deimmunization rate \deimmunizationRate. Let $T$ be the survival time of \contactProcess.
    Then for sufficiently large~$n$, it holds that $\E{T} \leq \big(\!\ln(\numberOfVertices)+2\big) (4\numberOfVertices^{ \deimmunizationRate}+1) \in \bigO{\numberOfVertices^{ \deimmunizationRate} \ln(\numberOfVertices)}$.
\end{restatable*}

Note that this bound is independent of~\infectionRate and that it results in a polynomial expected survival time as long as~\deimmunizationRate is at most constant with respect to~\numberOfVertices. Although we only prove an upper bound, our bound matches, up to a logarithmic factor, empirical investigations of the star~\cite{ferreira2016collective}, suggesting that our bound is almost tight.
Note that these experimental results consider the infection rate~\infectionRate to be constant in terms of \numberOfVertices, while our results apply for any \infectionRate.
Our results also show a behavior similar to the deterministic variant of the process considered by \textcite{saif2019epidemic}.

The analysis mainly relies on the method of investigating independent phases in which the center is not infected, bounding the probability of the infection process dying out during that time, as is common~\cite{BorgsCGS10Antidote,berger2005spread}.
A phase lasts at most until all leaves triggered their recovery at least once, which occurs in expectation after a time of about $\ln(\numberOfVertices)$.
Thus, if the center just recovered, it needs to become susceptible more quickly than that bound, as otherwise all leaves are recovered.
Since deimmunization triggers at rate~\deimmunizationRate, the probability that the center does not become susceptible in this time interval is about $\eulerE^{-\deimmunizationRate \ln \numberOfVertices}$, resulting in a probability of about $\numberOfVertices^{-\deimmunizationRate}$ that the infection dies out. Since these phases are independent, the infection process survives, in expectation, about $\numberOfVertices^{\deimmunizationRate}$ of these trials, each lasting about~$\ln(\numberOfVertices)$ time in expectation.

Note that the deimmunization rate and the state \emph{recovered} are important for this argument to hold.
Without this additional state, that is, in the SIS process, it is quite likely that the center becomes quickly reinfected before all leaves are not infected, which leads to an exponential expected survival time once~$\infectionRate \geq \numberOfVertices^{-1/2 + \varepsilon}$ in this setting~\cite{ganesh2005effect}, for all positive constants~$\varepsilon$.

\subsection{Expected survival time on expanders}

Before we state our main result, we formally introduce the notion of expansion we use for our results.
To this end, let $G = (V, E)$ be a graph with~$\numberOfVertices$ vertices $\{v_i\}_{i = 1}^{n}$, and let~$L$ be its normalized Laplacian, which is defined as
\begin{equation*}
    L_{i,j} =
    \begin{cases}
          1 & \text{if}\ i=j, \\
          -\frac{1}{\sqrt{\mathrm{deg}(v_i)\mathrm{deg}(v_j)}} & \text{if there is an edge between $v_i$ and $v_j$}, \\
          0 &\text{otherwise}.
    \end{cases}
\end{equation*}
Let~$L$ have eigenvalues $\lambda_1 \leq ... \leq \lambda_\numberOfVertices$.
The \emph{spectral expansion} of~$L$ is defined as $\spectralGap = \max_{i \geq 2}|1 - \lambda_i|$.
We call~$G$ an $(\numberOfVertices, (1\pm \degreeGap)\degree, \spectralGap)$-expander if and only if it has $\numberOfVertices$ vertices, a spectral expansion of $\spectralGap$ and only vertices with degree between $(1- \degreeGap)\degree$ and $(1+\degreeGap)\degree$.

As noted above, in contrast to stars, expanders feature many edges between arbitrary subsets of vertices. The key property we require for our results from $(\numberOfVertices, (1\pm \degreeGap)\degree, \spectralGap)$-expanders is that the number of edges between any two sets~$X$ and~$Y$ of vertices is close to $\frac{\degree}{\numberOfVertices} |X| |Y|$.

Our results hold for any expander~$G'$ that is subgraph of a graph~$G$ on which we analyze the SIRS process~$C$.
More formally, we define the \emph{projection}~$C'$ of~$C$ onto~$G'$ to be the process on~$G'$ such that, at each point in time, each vertex of~$G'$ in~$C'$ is in the same state as it is in~$C$. The survival time of a projected process is the first point in time that the projected process has no infected vertices. Given these definitions, our main result follows.

\begin{restatable*}{theorem}{SIRSClique}
    \label{thm:cliqueSIRS}
    Let $G$ be a graph, and let $G'$ be a subgraph of $G$ that is an $(\numberOfVertices,(1 \pm \degreeGap)\degree,\spectralGap)$-expander. Let $\degree\rightarrow\infty$ and $\spectralGap,\degreeGap\rightarrow0$ as $\numberOfVertices\rightarrow\infty$. Let \contactProcess be the SIRS process on~$G$ with infection rate~$\infectionRate$ and with constant deimmunization rate \deimmunizationRate. Further, let \contactProcess start with at least one infected vertex in $G'$ and no recovered vertices in~$G'$. Last, let~\contactProcessProjection be the projection of \contactProcess onto~$G'$, and let~$T$ be the survival time of~\contactProcessProjection.
    If $\infectionRate \geq \frac{\infectionConstant}{\degree}$ for a constant $\infectionConstant \in \R_{>1}$, then for sufficiently large \numberOfVertices, it holds that $\E{T} = 2^{\bigOmega{\numberOfVertices}}$.
\end{restatable*}

We note that \Cref{thm:cliqueSIRS} is almost tight with respect to the range of $\infectionRate$. \textcite[Theorem~3.1]{ganesh2005effect} show that the survival time of the SIS process is at most logarithmic in $n$ when the spectral radius of a graph is less than $1/\infectionRate$. Note that the spectral radius of a graph is upper bounded by the maximum degree of the graph. This results in a logarithmic expected survival time of the process on $(\numberOfVertices,(1 \pm \degreeGap)\degree,\spectralGap)$-expanders when $\infectionRate \leq \frac{1-\varepsilon}{d}$, for some constant~$\varepsilon$. Recall our discussion earlier in the introduction that the expected survival time of the SIS process is an upper bound to the expected survival time of the SIRS process. Hence, the expected survival time of the SIRS process for $\infectionRate\leq \frac{1-\varepsilon}{d}$ is at most logarithmic in~$n$ on $(\numberOfVertices,(1 \pm \degreeGap)\degree,\spectralGap)$-expanders.

The proof of \Cref{thm:cliqueSIRS} consists of two main parts. First, we prove that a linear number of vertices in~$G'$ becomes infected. Then, we show that the number of infected vertices stays linear for an expected exponential amount of time. For both parts, we make use of potential functions, which map the configuration of the process to a single real number that allows us to quantify how likely the process is to die out. In order to get the result on the projection of the process, we use that the influence of $G \setminus G'$ only increases the rate at which vertices in $G'$ get infected. In the considered configurations, this rate increase only helps the process to get into the desired region of the potential.

In more detail, the first part shows that the process reaches a configuration with at least $\varepsilon \numberOfVertices$ infected vertices with probability at least $\frac{1}{\numberOfVertices+2}$ (\Cref{lem:farFromEdgeSIRS}). Note that if this event does not occur, then the process might die out fast. For bounding the probability of this event, we use a fairly simple potential $\potentialIMinusR{\timeDiscrete}$ expressing the difference in the number of infected vertices minus $\varepsilon$ times the recovered vertices.
We show that $\potentialIMinusR{\timeDiscrete}$ is a submartigale. Applying the optional-stopping theorem to~$\potentialIMinusR{\timeDiscrete}$ concludes this first part of the proof.

In the second part, we define a more advanced potential function~\potentialSIRSClique{}{} (\Cref{def:lyapunovPotential}), which gets large when the number of infected vertices gets small. We show that there is a region of the potential in which the process is a strict supermartingale with a constant negative drift (\Cref{lem:constantDriftSIRS}). We show that in this region, higher infection rates decrease the drift (\Cref{lem:increasedInfection}). We then use the expansion properties of the base graph that guarantee that the infected vertices always have enough susceptible neighbors such that new vertices get infected and the potential decreases in expectation. This allows us to apply a concentration bound by \textcite{oliveto2011simplified} (\Cref{pre:NegativeDrift}) for strict supermartingales, known as \emph{negative-drift theorem}, based on an intricate theorem by \textcite{Hajek82HittingTime}.
The negative-drift theorem results in the lower exponential bound of the expected survival time.

Our definition of~\potentialSIRSClique{}{} is based on a Lyapunov function~\lyapunovHelper used by \textcite{korobeinikov2002lyapunov}, which they utilize in order to derive results on the global stability of the SIRS process via mean-field theory. The mean-field theory assumes a fully mixed graph, which roughly corresponds to a clique for our process. In order to show global stability, the authors show a negative drift towards an equilibrium point. However, this drift is~$0$ for some configurations in our setting, which is not small enough to apply the drift theorem. We adjust their function appropriately to create a region in the potential that has a sufficiently large negative drift. We also alter the analysis of the function to work in the stochastic process and on expander graphs.

\subsection{Expected survival time on special graph classes}

The generality of \Cref{thm:cliqueSIRS} makes it applicable to various other interesting graph classes.
The only requirement is that they contain an expander as a subgraph.
We illustrate this generality for two important random-graph models, namely, Erd\H{o}s--Rényi and hyperbolic random graphs.

\subsubsection*{\erdosGraphs}

The first random-graph model we are interested in is~$G_{\numberOfVertices,p}$---the classical random-graph model of \textcite{erdHos1959random}. The expansion properties of this model have been previously studied in literature. As \textcite[Theorem~1.2]{coja2007laplacian} shows, \erdosGraphs have a very small spectral expansion. Furthermore, due to Chernoff bounds, the vertex degrees in \erdosGraphs are tightly distributed around their average degree~$\degree$. Therefore, \erdosGraphs fulfill, with high probability, our definition of an $(\numberOfVertices,(1 \pm \degreeGap)\degree,\spectralGap)$-expander. This leads to the following corollary of \Cref{thm:cliqueSIRS}.

\begin{restatable*}{corollary}{ERGraphs}\label{cor:ER_graphs}
Let $G \sim G_{\numberOfVertices,p}$ be an \erdosGraph with $(\numberOfVertices-1)p \in \smallOmega{\ln\numberOfVertices}$. Consider the SIRS process $\contactProcess$ on $G$ with constant deimmunization rate \deimmunizationRate, and let $T$ be the survival time of $C$ when the process starts with at least one infected vertex.
If $\infectionRate\geq \frac{\infectionConstant}{\degree}$ for a constant $c \in \R_{>1}$, then $\E{T} = 2^{\bigOmega{\numberOfVertices}}$ asymptotically almost surely with respect to~$G$. If $\infectionRate\leq \frac{\infectionConstant}{\degree}$ for a constant $c \in (0, 1)$, then $\E{T} \in \bigO{\log \numberOfVertices}$ asymptotically almost surely with respect to~$G$.
\end{restatable*}

Comparing \Cref{cor:ER_graphs} with the respective result for the SIS process (cf. \cite[Theorem~$5.5$]{ganesh2005effect}) shows that the two processes, SIS and SIRS, behave similarly on \erdosGraphs.

\subsubsection*{Hyperbolic random graphs}\label{sec:Hyperbolic}

Many properties of complex real-world networks, such as the internet and social networks, are captured by hyperbolic random graphs~\cite{boguna2010sustaining,verbeek2014metric}. For this reason, since their introduction~\cite{KPKVB10}, hyperbolic random graphs are a very popular model in network theory that has been extensively studied (e.g.~\cite{gugelmann2012random,bode2015largest,muller2019diameter}). Therefore, hyperbolic random graphs provide a highly relevant structure for studying the survival time of the SIRS process. The exact definition of the model is not required to understand our results, hence we refer the reader to the work by \textcite{KPKVB10} for a formal definition.

The key parameter~$\hyperbolicExponent$ of a hyperbolic random graph controls the power-law exponent that the degree distribution follows. The interesting parameter range is $\hyperbolicExponent \in (2,3)$. Beyond this range, the graphs generated from this model lose key properties present in real-world networks. As the model commonly generates some very small disconnected components of a few vertices, the usual approach in literature is to focus on the giant component of the graph. Two of these properties are key for our results: the existence of a polynomial-sized clique as a subgraph and a polylogarithmic bound on the diameter of the giant component. Using these two properties, we identify the following parameter regime for the exponential expected survival time of the SIRS process on hyperbolic random graphs.

\begin{restatable*}{corollary}{HRGGraphs}
    \label{cor:SIRS_on_HRGs}
Let $G$ be a hyperbolic random graph with $\numberOfVertices$ vertices that follows a power-law degree distribution with exponent $\hyperbolicExponent \in (2,3)$, and let~$\contactProcess$ be the SIRS process on~$G$ with infection rate~$\infectionRate$ and with constant deimmunization rate~\deimmunizationRate. Further, let $\contactProcess$ start with at least one infected vertex in the giant component and no recovered vertices, and let $T$ be the survival time of $\contactProcess$.
Then there exists a constant $\infectionConstant \in \R_{>0}$ such that if $\infectionRate \geq \infectionConstant \numberOfVertices^{(\hyperbolicExponent-3)/2}$, then $\E{T} = 2^{\bigOmega{\numberOfVertices^{(3-\hyperbolicExponent)/2}}}$.
\end{restatable*}

\subsection{Outlook}\label{sec:intro_outlook}

Although our results cover already a great range of interesting graph classes, this article is just the first step to understanding the SIRS process more thoroughly.
Our analyses pose exciting new challenges for different scenarios, which we briefly delineate in the following.

Our upper bound of the expected survival time on stars (\Cref{lem:starSurvival}) is off from empirical results~\cite{ferreira2016collective} by a logarithmic factor.
This shows that there is potential for improvement in the analysis.
Ideally, proving a matching lower bound would answer the question for the exact expected survival time.

Combined, our results for stars (\Cref{lem:starSurvival}) and expanders (\Cref{thm:cliqueSIRS}) show that adding edges to a graph leads, eventually, from a polynomial expected survival time to an exponential.
However, it is not clear so far when this transition happens.
An interesting next step is to look into connected stars instead of single stars.
Connected stars appear as subgraphs in important real-world network models, such as the Chung--Lu~\cite{ChungL03ChungLuGraphs} or the preferential-attachment model~\cite{barabasi1999emergence}, motivating this research question.

With respect to expanders with vertex degrees concentrated around~$d$, our result (\Cref{thm:cliqueSIRS}) implies that~$1/d$ is the threshold for the infection rate~\infectionRate at which the expected survival time transitions from logarithmic to exponential.
However, our bounds require~\infectionRate to be bounded away from~$1/d$ by a constant.
It is not clear, given a value $\varepsilon \in \smallO{1}$, what happens if $\infectionRate = \frac{1 \pm \varepsilon}{d}$.
A more detailed analysis could provide insights into how rapidly the transition at the threshold occurs.

A different extension of our results is to consider deimmunization rates that are dependent on the graph size.
Comparing the behavior of the SIS and the SIRS process on stars suggests that an increased deimmunization rate leads to far longer expected survival times.
Thus, an interesting question is whether the survival time exhibits a threshold behavior with respect to the deimmunization rate.

Multi-dimensional potentials, as the one we use for the SIRS process on expanders, are rare in the analysis of stopping times of stochastic processes. Our approach draws inspiration from Lyapunov stability to devise a potential function for the stochastic process under study and then applies drift theory to convert this into a rigorous proof. Lyapunov functions are used in mean-field theory to show stable points of dynamical systems~\cite{lyapunov1992general}, and epidemic processes constitute only a glimpse of their applicability. We believe that our approach might inspire further rigorous results of determining stopping times of other stochastic processes, not limited to epidemic models.

\section{Preliminaries}
\label{sec:preliminaries}
We study the SIRS process, which is a continuous-time Markov chain on graphs in which the vertices change between different states, following events triggered by Poisson processes. We analyze how this process behaves asymptotically in the number of vertices \numberOfVertices of the graph. Especially, when we use big-O notation or refer to variables as constants, this is with respect to~$\numberOfVertices$.
When we use big-O notation inside of a term in a relation, this means that there exists a function from the big-O expression such that the relation holds, for example, the equation $a = 2^{\bigOmega{n}}$ means that there exists a function $f \in \bigOmega{n}$ such that $a = 2^{f(n)}$ holds.
If not stated otherwise, all variables we consider may depend on~$\numberOfVertices$. Whenever we talk about Poisson processes, we refer to one-dimensional Poisson point processes that output a random subset of the non-negative real numbers.

We first define our infection models and some related terms that we use throughout the paper. We then state the probabilistic tools we use in the proofs.

\subsection{Infection Processes}

Let $G=(V,E)$ be a graph with vertex set $V$ and edge set $E$. Further, let $\infectionRate,\deimmunizationRate \in \R_{>0}$. In the SIRS process, for each edge $e \in E$, we define a Poisson process $M_e$ with parameter $\infectionRate$, and for each vertex $v \in V$, we define the two Poisson processes $N_v$ with parameter $1$ and $O_v$ with parameter $\deimmunizationRate$. We refer to these processes as \emph{clocks}, and when an event occurs in one of them, we say that the relevant clock \emph{triggers}. We use $Z$ to denote the set of all of these clocks, that is, $Z = \left(\bigcup_{e \in E}{\{M_e\}}\right) \cup \left(\bigcup_{v \in V}{\{N_v,O_v\}}\right)$. Let $P$ denote the stochastic process in which all of the clocks in $Z$ evolve simultaneously and independently, starting at time 0. Note that almost surely there is no time point at which two clocks trigger at once. There are almost surely a countably infinite number of trigger times in $P$, which we index by the increasing sequence $\{\gamma_i\}_{i\in\N_{\geq0}}$, where $\gamma_0=0$.

A SIRS process $\contactProcess = (\contactProcess_\timePoint)_{\timePoint \in \R_{\geq 0}}$ has an underlying graph $G=(V,E)$, an infection rate $\infectionRate$, a deimmunization rate $\deimmunizationRate$, and an initial partition of $V$ into susceptible, infected, and recovered vertices with the respective sets $\susceptibleSet{0}$, $\infectedSet{0}$, and $\recoveredSet{0}$. At every time $\timePoint \in \R_{\geq 0}$, the configuration $\contactProcess_\timePoint$ is a partition of~$V$ into $\susceptibleSet{\timePoint}$, $\infectedSet{\timePoint}$, and $\recoveredSet{\timePoint}$. The configuration only changes at times in $P$. Let $i \in \N_{>0}$. We consider the following configuration transitions in $\gamma_i$:
\begin{itemize}
    \item If for some $e=\{u,v\}\in E$ we have $\gamma_i \in M_e$, $u \in \infectedSet{\gamma_{i-1}}$, and $v \in \susceptibleSet{\gamma_{i-1}}$, then $\susceptibleSet{\gamma_{i}} = \susceptibleSet{\gamma_{i-1}} \setminus \{v\}$, $\infectedSet{\gamma_{i}} = \infectedSet{\gamma_{i-1}} \cup \{v\}$, and $\recoveredSet{\gamma_{i}} = \recoveredSet{\gamma_{i-1}}$. We say that $v$ \emph{gets infected} at time point $\gamma_{i}$ by $u$.

    \item If for some $v \in V$ we have $\gamma_i \in N_v$ and $v \in \infectedSet{\gamma_{i-1}}$ then $\susceptibleSet{\gamma_{i}} = \susceptibleSet{\gamma_{i-1}}$, $\infectedSet{\gamma_{i}} = \infectedSet{\gamma_{i-1}} \setminus \{v\}$ and $\recoveredSet{\gamma_{i}} = \recoveredSet{\gamma_{i-1}} \cup \{v\}$. We say that $v$ \emph{recovers} at time point $\gamma_{i}$.

    \item If for some $v \in V$ we have $\gamma_i \in O_v$ and $v \in \recoveredSet{\gamma_{i-1}}$, then $\susceptibleSet{\gamma_{i}} = \susceptibleSet{\gamma_{i-1}} \cup \{v\}$, $\infectedSet{\gamma_{i}} = \infectedSet{\gamma_{i-1}}$ and $\recoveredSet{\gamma_{i}} = \recoveredSet{\gamma_{i-1}} \setminus \{v\}$. We say that $v$ \emph{gets susceptible} at time point $\gamma_{i}$.
\end{itemize}
If none of the above three cases occurs, the configuration of $\contactProcess$ at $\gamma_{i}$ is the same as the configuration of $\contactProcess$ at $\gamma_{i-1}$. Note that at all times between $\gamma_{i-1}$ and $\gamma_{i}$, $\contactProcess$ retains the same configuration as in $\gamma_{i-1}$.

In our proofs, we only consider the time points in $P$ at which the configuration changes. To this end, let $P'= \{\gamma_0\} \cup \{\gamma_i \mid i \in \N_{>0} \land \contactProcess_{\gamma_{i}} \neq \contactProcess_{\gamma_{i-1}}\}$. We index the times in $P'$ by the increasing sequence $\{\timeContinuous{i}\}_{i\in\N}$. For all $i \in \N$, we call $\timeContinuous{i}$ the $i$-th \emph{step} of the process.

If at any point in time no vertex is infected, then from that point onward, no vertex is infected.
We say that the infection \emph{dies out} or \emph{goes extinct} at the first (random) time~$T$ with $\infectedSet{T} = \emptyset$. We call~$T$ the \emph{survival time} of the SIRS process.

We only keep track of the number of vertices in each of the sets. To this end, we define for all $\timePoint \in \R_{\geq0}$ the random variables $\susceptibleContinuous{\timePoint} = |\susceptibleSet{\timePoint}|$, $\infectedContinuous{\timePoint} = |\infectedSet{\timePoint}|$, and $\recoveredContinuous{\timePoint} = |\recoveredSet{\timePoint}|$. These random variables change depending on the clocks in $P$. We say that an event \emph{happens at a rate of $r \in \R_{> 0}$} if and only if the set of clocks that cause this event when they trigger has a sum of rates equal to~$r$.

We define the \emph{projection} $C'$ of $C$ onto $G'$ as the process on~$G'$ such that, at each point in time, each vertex of~$G'$ in~$C'$ is in the same state as it is in~$C$. When considering such a projection, we use $\susceptibleContinuous{\timePoint}$, $\infectedContinuous{\timePoint}$, and $\recoveredContinuous{\timePoint}$ to only count the vertices of $C'$ in the corresponding state. Also~$\timeContinuous{i}$ only contains times at which the state of a vertex in $C'$ changes.
The survival time of a projected process is the first point in time that the projected process has no infected vertices.
Note that the survival time~$T'$ of~$C'$ is a lower bound for the survival time $T$ of $C$, as all infected vertices of $C'$ are also infected in $C$.

We use stochastic domination to transfer results from one random variable to another. We say that a random variable $(X_t)_{t \in \R}$ \emph{dominates} another random variable $(Y_t)_{t \in \R}$ if and only if there exists a coupling $(X'_t, Y'_t)_{t \in \R}$ in a way such that for all $t \in \R_{\geq 0}$ we have $X'_t \geq Y'_t$.

\subsection{Probabilistic Tools}

We use general concepts from probability theory (see for example \cite{feller1957introduction,mitzenmacher2017probability}). In addition, we use the following theorems.

We use the optional-stopping theorem for submartingales to bound the probability of reaching a specific configuration. For an event $E$, the symbol $\indicator{E}$ denotes the indicator random variable that is $1$ if $E$ is true and $0$ otherwise.

\begin{theorem}[Optional stopping {\cite[Theorem~13.2]{mitzenmacher2017probability}}]\label{pre:optionalStopping}
Let $(X_t)_{t \in \N}$ be a submartingale and $T$ a stopping time, both with respect to a filtration $(\filtration_t)_{t \in \N}$. Assume that the following two conditions hold:
\begin{enumerate}
\item $\E{T} < \infty$.
\item There is a $c \in \R$ such that for all $t\in \N$ we have $\E{|X_{t+1}-X_t|}[\filtration_t] \cdot \indicator{t<T} \leq c \cdot \indicator{t<T}$.
\end{enumerate}
Then $\E{X_T} \geq \E{X_0}$.
\end{theorem}

We use the following theorem in order to show an exponential expected survival time for the SIRS process. We state it in a fashion that better suits our purposes.

\begin{theorem}[Negative drift {\cite[Theorem~$4$]{oliveto2011simplified}~\cite{OlivetoW12NegativeDriftErratum}}]\label{pre:NegativeDrift}
Let $(X_t)_{t \in \N}$ be a random process over~\R, adapted to a filtration $(\filtrationContinuous{t})_{t \in \N}$. Let there be an interval $[a,b] \subseteq \R$, two constants $\delta,\varepsilon \in \R_{>0}$ and, possibly depending on $l \defeq b-a$, a function $r(l)$ satisfying $1 \leq r(l) \in \smallO{l/\log(l)}$. Let $T = \inf\{ t \in \N \mid X_t \geq b \}$. Suppose that for all $t \in \N$ the following two conditions hold:
\begin{enumerate}
    \item $\E{X_{t+1}-X_t}[\filtrationContinuous{t}] \cdot \indicator{a < X_t < b} \leq -\varepsilon \cdot \indicator{a < X_t < b}$.
    \item For all $j \in \R_{\geq 0}$ we have $\Pr{|X_{t+1}-X_t|\geq j}[ \filtrationContinuous{t}] \cdot \indicator{t<T} \leq \frac{r(l)}{(1+\delta)^j} \cdot \indicator{t<T}.$
\end{enumerate}
Then there exists a constant $c \in \R_{>0}$ such that
\begin{align*}
    \Pr{T \leq 2^{cl/r(l)}}[\filtrationContinuous{0}][\big] \cdot \indicator{X_0 \leq a} &= 2^{-\bigOmega{l/r(l)}} \cdot \indicator{X_0 \leq a}.\qedhere
\end{align*}
\end{theorem}

The following theorem bounds the expected value of the maximum of $n$ exponentially distributed random variables.

\begin{theorem}[{\cite[Lemma~2.10]{mitzenmacher2017probability}}]\label{pre:maxExponential}
Let $n \in \N_{> 0}$, and let $\{X_i\}_{i \in [n]}$ be independent random variables that are each exponentially distributed with parameter $\lambda \in \R_{> 0}$. Let $m = \max_{i \in [n]} X_i$, and let $H_n$ be the $n$-th harmonic number. Then
\begin{align*}
\E{m} &= \frac{H_n}{\lambda} < \frac{1 + \ln(n)}{\lambda}.\qedhere
\end{align*}
\end{theorem}

We use the following version of Wald's equation, which does not require the addends to be independent.

\begin{theorem}[Generalized Wald's equation~{\cite[Theorem~$5$]{DoerrK22WaldsEquation}}]\label{pre:wald}
Let $c,c' \in \R$, and let $(X_t)_{t \in \N}$ be a random process over $\R_{\geq c}$ such that $\sum_{i \in [S]}{X_i}$ has a finite expectation. Furthermore, let $(\filtration_t)_{t \in \N}$ be a filtration, and let $S$ be a stopping time with respect to $\filtration$.
If  for all $i \in \N$, it holds that $\E{X_{i+1}}[\filtration_i] \leq c'$, then
\begin{align*}
&\E{\sum\nolimits_{i \in [S]} {X_i}}[\filtration_0] = \E{\sum\nolimits_{i \in [S]} {\E{X_i} [ \filtration_{i-1}]}}[\filtration_0].\qedhere
\end{align*}
\end{theorem}

\subsection{Expander graphs}

There are many notions of how to define expander graphs. We use algebraic expanders in which all but one of the eigenvalues of the normalized Laplacian of the graph are very close to~$1$. These graphs have some nice properties that let us bound the number of edges between infected and susceptible vertices.
Formally, let $G = (V, E)$ be a graph with~$\numberOfVertices$ vertices $\{v_i\}_{i = 1}^{n}$, and let~$L$ be its normalized Laplacian, which is defined as
\begin{equation*}
    L_{i,j} =
    \begin{cases}
        1 & \text{if}\ i=j, \\
        -\frac{1}{\sqrt{\mathrm{deg}(v_i)\mathrm{deg}(v_j)}} & \text{if there is an edge between $v_i$ and $v_j$}, \\
        0 &\text{otherwise}.
    \end{cases}
\end{equation*}
Let~$L$ have eigenvalues $\lambda_1 \leq ... \leq \lambda_\numberOfVertices$.
The \emph{spectral expansion} of~$L$ is defined as $\spectralGap = \max_{i \geq 2}|1 - \lambda_i|$.
We call~$G$ an $(\numberOfVertices, (1\pm \degreeGap)\degree, \spectralGap)$-expander if and only if it has $\numberOfVertices$ vertices, a spectral expansion of $\spectralGap$ and only vertices with degree between $(1- \degreeGap)\degree$ and $(1+\degreeGap)\degree$.

For two vertex sets $X,Y \subseteq V$, let $\edges{X}{Y}$ denote the number of edges between $X$ and $Y$. Using this notation, we have the following theorem

\begin{theorem}[{\cite[Theorem~$5.2$]{book:694330}}]\label{pre:expanderGeneral}
Let $G=(V,E)$ be a graph and let $X,Y \subseteq V$. Then
\begin{align*}
    \left||\edges{X}{Y}|-\frac{\vol{X} \cdot \vol{Y}}{\vol{V}}\right|&\leq \spectralGap \cdot \frac{\sqrt{\vol{X} \vol{\overline{X}} \vol{Y} \vol{\overline{Y}}}}{\vol{V}}.\qedhere
\end{align*}
\end{theorem}

Applying \Cref{pre:expanderGeneral} to expanders, we get the following two corollaries..

\begin{corollary}\label{pre:expanderComplement}
Let $G=(V,E)$ be a $(\numberOfVertices, (1\pm \degreeGap)\degree, \spectralGap)$-expander, and let $X \subseteq V$. Then
\begin{align*}
    |\edges{X}{\overline{X}}|&\geq (1-\spectralGap)(1- 3 \degreeGap)\degree \frac{|X|\cdot |\overline{X}|}{n}.\qedhere
\end{align*}
\end{corollary}

\begin{proof}
Because the vertex degrees of all vertices in $G$ are bounded, we know that for each $S \subseteq V$ holds $(1-\degreeGap)\degree |S| \leq \vol{S} \leq (1-\degreeGap)\degree |S|$. Plugging that into the result of \Cref{pre:expanderGeneral} gives us
\begin{align*}
    |\edges{X}{\overline{X}}|&\geq \frac{\vol{X} \cdot \vol{\overline{X}}}{\vol{V}}- \spectralGap \cdot \frac{\sqrt{\vol{X} \vol{\overline{X}} \vol{\overline{X}} \vol{\overline{\overline{X}}}}}{\vol{V}}\\
    &= (1- \spectralGap)\frac{\vol{X} \cdot \vol{\overline{X}}}{\vol{V}}\\
    &\geq (1-\spectralGap) \frac{(1-\degreeGap)\degree |X|\cdot (1-\degreeGap)\degree |\overline{X}|}{(1+\degreeGap)\degree \numberOfVertices}\\
    &\geq (1-\spectralGap)(1- 3 \degreeGap)\degree \frac{|X|\cdot |\overline{X}|}{n}.\qedhere
\end{align*}
\end{proof}

\begin{corollary}\label{pre:expanderbroad}
Let $G=(V,E)$ be a $(\numberOfVertices, (1\pm \degreeGap)\degree, \spectralGap)$-expander, and let $X,Y \subseteq V$.
If $\degreeGap \leq 1/5$, then
\begin{align*}
    \left||\edges{X}{Y}| - \degree\frac{|X|\cdot|Y|}{\numberOfVertices}\right|&\leq 4 \degreeGap \degree\frac{|X|\cdot|Y|}{\numberOfVertices} + 2\spectralGap\degree\sqrt{|X|\cdot |Y|}.\qedhere
\end{align*}
\end{corollary}

\begin{proof}
Because the vertex degrees of all vertices in $G$ are bounded, we know that for each $S \subseteq V$ holds $(1-\degreeGap)\degree |S| \leq \vol{S} \leq (1+\degreeGap)\degree |S|$. \Cref{pre:expanderGeneral} gives us both an upper and a lower bound for $|\edges{X}{Y}|-\frac{\vol{X} \cdot \vol{Y}}{\vol{V}}$. We solve them for $|\edges{X}{Y}|$ and bound them separately using that $\degreeGap \leq 1/5$.
\begin{align*}
    |\edges{X}{Y}|&\geq \frac{\vol{X} \cdot \vol{Y}}{\vol{V}} - \spectralGap \cdot \frac{\sqrt{\vol{X} \vol{\overline{X}} \vol{Y} \vol{\overline{Y}}}}{\vol{V}}\\
    &\geq \frac{\vol{X} \cdot \vol{Y}}{\vol{V}} - \spectralGap \cdot \sqrt{\vol{X} \vol{Y}}\\
    &\geq \frac{(1-\degreeGap)\degree |X| \cdot (1-\degreeGap)\degree |Y|}{(1+\degreeGap)\degree \numberOfVertices} - \spectralGap (1+\degreeGap)\degree \sqrt{|X| \cdot |Y|}\\
    &\geq (1- 4\degreeGap)\degree \frac{|X| \cdot |Y|}{\numberOfVertices} - 2 \spectralGap \degree \sqrt{|X| \cdot |Y|} \textrm{ as well as}\\
    |\edges{X}{Y}|&\leq \frac{\vol{X} \cdot \vol{Y}}{\vol{V}} + \spectralGap \cdot \frac{\sqrt{\vol{X} \vol{\overline{X}} \vol{Y} \vol{\overline{Y}}}}{\vol{V}}\\
    &\leq \frac{\vol{X} \cdot \vol{Y}}{\vol{V}} + \spectralGap \cdot \sqrt{\vol{X} \vol{Y}}\\
    &\leq \frac{(1+\degreeGap)\degree |X| \cdot (1+\degreeGap)\degree |Y|}{(1-\degreeGap)\degree \numberOfVertices} + \spectralGap (1+\degreeGap)\degree \sqrt{|X| \cdot |Y|}\\
    &= \left(1+ \frac{3\degreeGap + \degreeGap^2}{1-\degreeGap}\right)\degree \frac{|X| \cdot |Y|}{\numberOfVertices} + \spectralGap (1+\degreeGap)\degree \sqrt{|X| \cdot |Y|}\\
    &\leq (1+ 4\degreeGap)\degree \frac{|X| \cdot |Y|}{\numberOfVertices} + 2 \spectralGap \degree \sqrt{|X| \cdot |Y|}.
\end{align*}

Subtracting $\degree \frac{|X| \cdot |Y|}{\numberOfVertices}$ from both inequalities and combining them proves the corollary.
\end{proof}

\section{SIRS on Stars}
\label{sec:stars}
We show that the expected survival time of the SIRS process on stars is bounded from above by a polynomial in the number of vertices that is independent of the infection rate (\Cref{lem:starSurvival}). To this end, we bound the number of times that the center gets infected and the time between two infections of the center. We use that while the center is not infected, no leaf gets infected. Hence, if all of the leaves recover before the center gets susceptible after it recovered, the infection dies out.

We first bound the expected time that it takes for all of the leaves to recover.
We refer to each clock at a vertex whose rate is the recovery (of~$1$) rate as \emph{recovery clock}.

\begin{lemma}\label{lem:starHealthyPhaseDuration}
Let $G$ be a star with $\numberOfVertices \in \N_{>0}$ leaves, and let~\contactProcess be a SIRS process on~$G$ with infection rate~\infectionRate and with deimmunization rate~\deimmunizationRate. Let~$T$ be the time that it takes for all recovery clocks of the leaves to trigger at least once. Then $\E{T} \leq \ln(\numberOfVertices) +1$.
\end{lemma}
\begin{proof}
The star has \numberOfVertices leaves, which all have a clock that recovers them at a rate of~$1$. For each of the clocks, the time until the first trigger happens is exponentially distributed with parameter~$1$. Hence, $T$ is calculated as the maximum of the $\numberOfVertices$ exponential distributions of the independent clocks. By \Cref{pre:maxExponential}, $\E{T} \leq \ln(\numberOfVertices) + 1$.
\end{proof}

We now use \Cref{lem:starHealthyPhaseDuration} to bound the time it takes from one infection of the center until it gets infected again or until the infection dies out.

\begin{lemma}\label{lem:starPhaseDuration}
Let $G$ be a star with $\numberOfVertices \in \N_{>0}$ leaves, and let \contactProcess be a SIRS process on $G$ with infection rate \infectionRate and with deimmunization rate \deimmunizationRate. Let $t_0 \in \R_{\geq 0}$ be a time at which the infection has not died out yet. Further, let $T \in \R_{\geq 0}$ be the first time after $t_0$ at which either the center gets infected after being susceptible or the infection dies out. Then $\E{T - t_0} \leq \ln(\numberOfVertices) + 2$.
\end{lemma}

\begin{proof}
If the center starts infected, in order for either the center to get infected again after being susceptible or the infection to die out, the center has to recover first. Let $T' \in \R$ be the first time after $t_0$ at which the center recovers. As all vertices recover at a rate of~$1$, the random variable $T' - t_0$ is exponentially distributed with a parameter of~$1$.

Between $T'$ and $T$, no leaf gets infected, as the center is not infected and all edges are incident to the center. Hence, when all recovery clocks of the leaves trigger in this time interval at least once, the infection dies out. Therefore, the last point in time after~$T'$ at which any of these recovery clocks trigger is an upper bound for~$T$. By \Cref{lem:starHealthyPhaseDuration}, the expected time for this last trigger to happen is at most $\ln(\numberOfVertices) +1$. That gives us
\begin{align*}
    \E{T -t_0} &= \E{T - T' + T' - t_0}\\
    &= \E{T-T'} + \E{T' - t_0}\\
    &\leq \ln(\numberOfVertices) + 2.\qedhere
\end{align*}
\end{proof}

Next, we bound the probability from below that when starting with an infected center, the infection dies out before the center gets infected again. We use this later to get an upper bound on the number of times that the center gets infected in total.

\begin{lemma}\label{lem:starDieOutProbability}
Let $G$ be a star with $\numberOfVertices \in \N_{>0}$ leaves, and let \contactProcess be a SIRS process on $G$ with infection rate \infectionRate and with deimmunization rate \deimmunizationRate. Let $t_0 \in \R_{\geq 0}$ be a time at which the center is infected. Further, let $E_0$ be the event that the infection dies out after $t_0$ before the center gets infected again (after being recovered in between). Then for sufficiently large \numberOfVertices, it holds that $\Pr{E_0} \geq \frac{1}{4} \numberOfVertices^{- \deimmunizationRate}$.
\end{lemma}

\begin{proof}
In order for either the center to get infected again after being susceptible or for the infection to die out, the center has to recover first. Let $t_1 \in \R$ be the first time after $t_0$ at which the center recovers. As long as the center is in the recovered state, no vertex gets infected, as all edges of the graph are incident to the center. If all leaves recover before the center gets susceptible, the infection dies out. In order to bound the probability of this event, we consider the first time $T \in \R$ after $t_1$ at which the center gets susceptible, and we also consider the first time $T' \in \R$ after $t_1$ at which all of the recovery clocks of the leaves trigger at least once in the interval $(t_1,T']$. In particular, we use that all leaves recover before the center gets susceptible if $T'-t_1 <  \ln(\numberOfVertices)$ and $T-t_1 \geq  \ln(\numberOfVertices)$.

Each vertex recovers after a time that is exponentially distributed with parameter~$1$. As $T'$ is the first time after $t_1$ at which all of the recovery clocks of the leaves trigger at least once in the interval $(t_1,T']$, it is the maximum of~$\numberOfVertices$ exponentially distributed random variables. In order for $T' - t_1 < \ln(\numberOfVertices)$, all of those random variables have to be smaller than $\ln(\numberOfVertices)$. As all of them are independent, we get that, for sufficiently large $\numberOfVertices$,
\begin{align*}
    \Pr{T' - t_1 < \ln(\numberOfVertices)} &= \Pr{\text{Exp}(1)<\ln(\numberOfVertices)}^\numberOfVertices\\
    &= \left(1-\eulerE^{-1 \ln(\numberOfVertices)}\right)^\numberOfVertices\\
    &= \left(1 - \frac{1}{n}\right)^n\\
    &\geq \frac{1}{4}.
\end{align*}

All vertices lose their immunity at a rate of \deimmunizationRate. Hence, $T-t_1$ is exponentially distributed with parameter \deimmunizationRate. Using the exponential probability distribution, we get
\begin{align*}
    \Pr{T- t_1 \geq  \ln(\numberOfVertices)} &= \eulerE^{- \deimmunizationRate \ln(\numberOfVertices)}\\
    &= \numberOfVertices^{- \deimmunizationRate}.
\end{align*}

Now using the fact that the infection dies out when all leaves recover before the center gets susceptible and that $T- t_1$ and $T' -t_1$ are independent, we get
\begin{align*}
    \Pr{E_0} &\geq \Pr{T'-t_1 < T - t_1}\\
    &\geq \Pr{T'-t_1 <  \ln(\numberOfVertices) \land T-t_1 \geq  \ln(\numberOfVertices)}\\
    &= \Pr{T'-t_1 <  \ln(\numberOfVertices)} \cdot \Pr{ T-t_1 \geq  \ln(\numberOfVertices)}\\
    &\geq \frac{1}{4}\numberOfVertices^{- \deimmunizationRate}. \qedhere
\end{align*}
\end{proof}

Using the previous bounds, we now derive an upper bound on the expected survival time of a SIRS process on a star.

\StarSurvival

\begin{proof}
Let $S$ be the random variable that counts the number of times that the center gets infected before the infection dies out. For all $i \in \N_{\leq S+1}$, let $X_i$ be the $i$-th time at which either the center gets infected or the infection dies out (we define $X_0 = 0$). It then holds that $T = X_{S+1} = \sum_{i=0}^{S}{X_{i+1}-X_i}$. We aim to bound the expectation of this value using the generalized Wald's equation (\Cref{pre:wald}).

Let $(\filtrationContinuous{\timePoint})_{\timePoint \in \R_{\geq 0}}$ be the natural filtration of \contactProcess. By \Cref{lem:starPhaseDuration}, it holds for all $i \in \N_{\leq S}$ that $0 \leq \E{X_{i+1}-X_i}[\filtrationContinuous{X_i}] \leq \ln(\numberOfVertices) +2$. Hence, the expectations of all of the summed random variables are bounded. By \Cref{lem:starDieOutProbability}, for all $i \in \N_{\geq 1}$, the $i$-th infection of the center has a probability of at least $\frac{1}{4} \numberOfVertices^{- \deimmunizationRate}$ to be the last one if there is an $i$-th infection of the center. Therefore, $S$ is dominated by a geometrically distributed random variable $A \sim \text{Geom}(\frac{1}{4} \numberOfVertices^{- \deimmunizationRate})$. Hence, $\sum_{i=0}^{S}{X_{i+1}-X_i}$ is integrable. By \Cref{pre:wald}, we get
\begin{align*}
    \E{T}[\filtrationContinuous{0}] &= \E{\sum_{i=0}^{S}{X_{i+1}-X_i}}[\filtrationContinuous{0}]\\
    &= \E{\sum_{i=0}^{S}{\E{X_{i+1}-X_i}[\filtrationContinuous{X_i}]}}[\filtrationContinuous{0}]\\
    &\leq \E{\sum_{i=0}^{S}{\ln(\numberOfVertices)+2}}[\filtrationContinuous{0}]\\
    &= (\ln(\numberOfVertices)+2) \E{\sum_{i=0}^{S}{1}}[\filtrationContinuous{0}]\\
    &\leq (\ln(\numberOfVertices)+2) (4\numberOfVertices^{ \deimmunizationRate}+1).\qedhere
\end{align*}
\end{proof}

\section{SIRS on Expanders}
\label{sec:SIRS_clique}
We consider the SIRS process on graphs that have expanders as subgraphs. In particular, we show an exponential expected survival time for the projection of the SIRS process onto the expander when the deimmunization rate is constant and the infection rate is sufficiently high (\Cref{thm:cliqueSIRS}). Note that the exponential expected survival time and the required infection rate depend only on the size and vertex degrees of the expander. In \Cref{sec:SIRSCliqueBasic}, we begin by analyzing basic properties of the process, such as the transition rates between all of the states.

In \Cref{sec:SIRSCliqueThreshold}, we show that the expected survival time of the considered SIRS processes is exponential if $\infectionRate \geq \frac{\infectionConstant}{\degree}$ for a constant $\infectionConstant \in \R_{>1}$. We first prove that the process reaches a configuration with at least $\varepsilon \numberOfVertices$ infected vertices with sufficiently high probability. We then provide a lower bound for the expected survival time starting at such a configuration. To this end, we define a potential over the configuration space that has in a specific region a constant negative drift away from the configuration with no infected vertices. We then translate this region into bounds for the potential, allowing us to apply the negative-drift theorem (\Cref{pre:NegativeDrift}) to get an exponential expected survival time.

\subsection{The SIRS Process}\label{sec:SIRSCliqueBasic}

Let $G=(V,E)$ be a graph and let $G'= (V',E')$ be a subgraph of $G$ that is an $(\numberOfVertices,(1\pm\degreeGap)\degree,\spectralGap)$-expander. Let \contactProcess be a SIRS process  with infection rate $\infectionRate \geq \frac{\infectionConstant}{\degree}$ for a constant $\infectionConstant \in \R_{>1}$ and deimmunization rate \deimmunizationRate on $G$. Consider the projection \contactProcessProjection of \contactProcess onto $G'$. We define for all $\timeDiscrete \in \N$ the random variable $\susceptibleDiscreteShifted{\timeDiscrete} = \susceptibleDiscrete{\timeDiscrete} + \frac{\deimmunizationRate}{\infectionConstant}\numberOfVertices$. We use $\susceptibleDiscreteShifted{\timeDiscrete}$ to define the potential later. Roughly, using $\susceptibleDiscreteShifted{\timeDiscrete}$ instead of $\susceptibleDiscrete{\timeDiscrete}$ has the effect that changes of $\susceptibleDiscrete{\timeDiscrete}$ have a lower impact on the potential. Note that, at all times $t$, it holds that $\susceptibleDiscrete{\timeDiscrete} + \infectedDiscrete{\timeDiscrete} + \recoveredDiscrete{\timeDiscrete} = \numberOfVertices$, since every vertex of $G$ is always in exactly one of these three sets. Additionally, $\susceptibleDiscreteShifted{\timeDiscrete} + \infectedDiscrete{\timeDiscrete} + \recoveredDiscrete{\timeDiscrete}  = \numberOfVertices + \frac{\deimmunizationRate}{\infectionConstant}\numberOfVertices = \numberOfVerticesShifted$.

For all $\timeDiscrete \in \N$, one of the following four events occurs at step $\timeDiscrete+1$ (i.e., $\tau_{t+1}$): either a susceptible vertex is infected through an edge outside of $G'$, which we call \eventInfectionOutside{\timeDiscrete};
or a susceptible vertex is infected through an edge inside of $G'$, which we call \eventInfection{\timeDiscrete}; or an infected vertex recovers in the event \eventRecover{\timeDiscrete}; or a recovered vertex loses its immunity, which we call \eventSusceptible{\timeDiscrete}.

For each time point \timeContinuous{\timeDiscrete}, let $E_{\timeContinuous{\timeDiscrete}}(I,S)$ be the number of edges between the infected and the susceptible vertices in $G'$. At the time point \timeContinuous{\timeDiscrete}, vertices get infected by other vertices via edges inside~$G'$ at a rate of $\rateInfection{\timeDiscrete} = \infectionRate E_{\timeContinuous{\timeDiscrete}}(I,S)$, because every infected vertex infects each susceptible vertices at a rate of \infectionRate. Vertices recover from an infection at a rate of $\rateRecover{\timeDiscrete} = \infectedDiscrete{\timeDiscrete}$ and get susceptible at a rate of $\rateSusceptible{\timeDiscrete} = \deimmunizationRate \recoveredDiscrete{\timeDiscrete}$. As we only consider the states of the vertices in~$G'$, we cannot calculate the rate $\rateInfectionOutside{\timeDiscrete}$ at which susceptible vertices get infected through edges outside of $G'$, we only know that it is non-negative. Now let $\rateTotalOutside{\timeDiscrete} = \rateInfectionOutside{\timeDiscrete} +  \rateInfection{\timeDiscrete} + \rateRecover{\timeDiscrete} + \rateSusceptible{\timeDiscrete}$. We get
\begin{align*}
    \probabilityInfectionOutside{\timeDiscrete} &= \Pr{\eventInfectionOutside{\timeDiscrete}} = \frac{\rateInfectionOutside{\timeDiscrete}}{\rateTotalOutside{\timeDiscrete}} \geq 0,\\
    \probabilityInfection{\timeDiscrete} &= \Pr{\eventInfection{\timeDiscrete}} = \frac{\rateInfection{\timeDiscrete}}{\rateTotalOutside{\timeDiscrete}} = \frac{\infectionRate E_{\timeContinuous{\timeDiscrete}}(I,S)}{\rateTotalOutside{\timeDiscrete}},\\
    \probabilityRecover{\timeDiscrete} &= \Pr{\eventRecover{\timeDiscrete}} = \frac{\rateRecover{\timeDiscrete}}{\rateTotalOutside{\timeDiscrete}} = \frac{\infectedDiscrete{\timeDiscrete}}{\rateTotalOutside{\timeDiscrete}},\textrm{ and }\\
    \probabilitySusceptible{\timeDiscrete} &= \Pr{\eventSusceptible{\timeDiscrete}} = \frac{\rateSusceptible{\timeDiscrete}}{\rateTotalOutside{\timeDiscrete}} = \frac{\deimmunizationRate \recoveredDiscrete{\timeDiscrete}}{\rateTotalOutside{\timeDiscrete}}.
\end{align*}

Note that we only consider these probabilities in configurations in which at least one vertex is infected, hence $\rateTotalOutside{\timeDiscrete} \neq 0$ and the above probabilities are well-defined. We now define
\begin{align*}
    \infectedEquilibrium &= \frac{\deimmunizationRate (\infectionConstant-1)}{(1 + \deimmunizationRate) \infectionConstant}\numberOfVertices.
\end{align*}

This value is the number of infected vertices in an equilibrium configuration of a SIRS process on a clique with $\numberOfVertices$ vertices and an infection rate of $\frac{\infectionConstant}{\numberOfVertices}$. We use this value as a clique and the expanders we consider behave very similarly, thus, $\infectedEquilibrium$ is a good estimate for the number of infected vertices that $C$ tends to have on $G'$.

\subsection{Exponential survival time}\label{sec:SIRSCliqueThreshold}

We now show that the infection becomes epidemic if $\infectionRate \geq \frac{\infectionConstant}{\degree}$ for a constant $\infectionConstant \in \R_{>1}$.
We start by proving that, when starting with one infected vertex inside of the expander, the infection reaches a configuration with at least $\varepsilon \numberOfVertices$ infected vertices with sufficiently large probability.

\begin{lemma}\label{lem:farFromEdgeSIRS}
    Let $G$ be a graph, and let $G'$ be a subgraph of $G$ that is a $(\numberOfVertices,(1 \pm \degreeGap)\degree,\spectralGap)$-expander. Let $\degree\rightarrow\infty$ and $\spectralGap,\degreeGap\rightarrow0$ as $\numberOfVertices\rightarrow\infty$. Let \contactProcess be a SIRS process on $G$ with infection rate~$\infectionRate$ and with constant deimmunization rate \deimmunizationRate. Also let \contactProcess start with at least one infected vertex in $G'$ and no recovered vertices in~$G'$. Consider the projection \contactProcessProjection of \contactProcess onto $G'$.
    If $\infectionRate \geq \frac{\infectionConstant}{\degree}$ for a constant $\infectionConstant \in \R_{>1}$, then there exists an $\varepsilon \in \R_{>0}$ such that for sufficiently large \numberOfVertices, the probability that there exists a $\timeDiscrete \in \N$ with $\infectedDiscrete{\timeDiscrete} \geq \varepsilon \numberOfVertices$ is at least $\frac{1}{\numberOfVertices+2}$.
\end{lemma}

\begin{proof}
    Let $\infectionConstantMinusOne=\infectionConstant-1$. Note that \infectionConstantMinusOne is positive because $\infectionConstant>1$. Let $\varepsilon_H, \varepsilon_S\in \R_{>0}$ be a constants that we specify later. We define for all $\timeDiscrete \in \N$ the potential  $\potentialIMinusR{\timeDiscrete} =\potentialFunctionIMinusR[\infectedDiscrete{\timeDiscrete},\recoveredDiscrete{\timeDiscrete}]= \infectedDiscrete{\timeDiscrete} - \varepsilon_H\recoveredDiscrete{\timeDiscrete}$. Additionally, we define the stopping time $T = \inf\{\timeDiscrete \in \N \mid \potentialIMinusR{\timeDiscrete} \leq 0 \lor \susceptibleDiscrete{\timeDiscrete} < (1-\varepsilon_S)\numberOfVertices\}$ and the natural filtration $(\filtrationContinuous{\timePoint})_{\timePoint \in \R_{\geq 0}}$ of \contactProcess. We aim to show that $(\potentialIMinusR{\timeDiscrete})_{\timeDiscrete \in \N}$ is a sub-martingale until $T$. This allows us to apply the optional-stopping theorem (\Cref{pre:optionalStopping}) to bound $\E{\potentialIMinusR{T}}$ from below. The law of total expectation then yields a lower bound of $\frac{1}{\numberOfVertices+2}$ for $\Pr{\potentialIMinusR{T} > 0}$. We conclude the proof by showing that if $\potentialIMinusR{T} > 0$, then $\infectedDiscrete{T} \geq \varepsilon \numberOfVertices$.

    We first bound $\rateInfection{\timeDiscrete}$ using \Cref{pre:expanderComplement} for all times $\timeDiscrete < T$. We get
    \begin{align*}
        \rateInfection{\timeDiscrete} &= \infectionRate \edges{I}{S}\\
        &\geq\infectionRate\left(\edges{I+R}{S}-(1 + \degreeGap)\degree \recoveredDiscrete{\timeDiscrete}\right)\\
        &\geq \infectionRate\left((1-\spectralGap)(1-3\degreeGap)\frac{\degree (\infectedDiscrete{\timeDiscrete}+\recoveredDiscrete{\timeDiscrete})\susceptibleDiscrete{\timeDiscrete}}{\numberOfVertices}-(1 + \degreeGap)\degree \recoveredDiscrete{\timeDiscrete}\right)\\
        &\geq \infectionRate\left((1-\spectralGap)(1-3\degreeGap)(1-\varepsilon_S)\degree (\infectedDiscrete{\timeDiscrete}+\recoveredDiscrete{\timeDiscrete})-(1+\degreeGap)\degree \recoveredDiscrete{\timeDiscrete}\right)\\
        &\geq \frac{\infectionConstant}{\degree}\left((1-\spectralGap-3\degreeGap-\varepsilon_S)\degree (\infectedDiscrete{\timeDiscrete}+\recoveredDiscrete{\timeDiscrete})-(1+\degreeGap)\degree \recoveredDiscrete{\timeDiscrete}\right)\\
        &\geq \infectionConstant \infectedDiscrete{\timeDiscrete} - (\spectralGap+4\degreeGap+\varepsilon_s)\infectionConstant(\infectedDiscrete{\timeDiscrete}+\recoveredDiscrete{\timeDiscrete}).
    \end{align*}

    We now bound for all $\timeDiscrete \in \N$ the drift $\E{(\potentialIMinusR{\timeDiscrete+1}-\potentialIMinusR{\timeDiscrete}) \cdot \indicator{t<T}} [\filtrationDiscrete{\timeDiscrete}]$. To improve readability, we omit the multiplicative $\indicator{t<T}$ in all of the terms.
    \begin{align*}
        \E{\potentialIMinusR{\timeDiscrete+1}-\potentialIMinusR{\timeDiscrete}} [\filtrationDiscrete{\timeDiscrete}] &= ( \probabilityInfection{\timeDiscrete}+\probabilityInfectionOutside{\timeDiscrete}) \cdot \left(\potentialFunctionIMinusR[\infectedDiscrete{\timeDiscrete}+1,\recoveredDiscrete{\timeDiscrete}]-\potentialIMinusR{\timeDiscrete}\right)\\
        &\quad+ \probabilityRecover{\timeDiscrete} \left(\potentialFunctionIMinusR[\infectedDiscrete{\timeDiscrete}-1,\recoveredDiscrete{\timeDiscrete}+1]-\potentialIMinusR{\timeDiscrete}\right) + \probabilitySusceptible{\timeDiscrete} \left(\potentialFunctionIMinusR[\infectedDiscrete{\timeDiscrete},\recoveredDiscrete{\timeDiscrete}-1]-\potentialIMinusR{\timeDiscrete}\right)\\
        &= \probabilityInfection{\timeDiscrete} + \probabilityInfectionOutside{\timeDiscrete} - \probabilityRecover{\timeDiscrete}(1+\varepsilon_H) + \probabilitySusceptible{\timeDiscrete} \varepsilon_H\\
        &\geq \probabilityInfection{\timeDiscrete} - \probabilityRecover{\timeDiscrete}(1+\varepsilon_H) + \probabilitySusceptible{\timeDiscrete} \varepsilon_H\\
        &\geq \left(\infectionConstant \infectedDiscrete{\timeDiscrete} - (\spectralGap+4\degreeGap+\varepsilon_s)\infectionConstant(\infectedDiscrete{\timeDiscrete}+\recoveredDiscrete{\timeDiscrete}) - \infectedDiscrete{\timeDiscrete}(1+\varepsilon_H) + \deimmunizationRate \recoveredDiscrete{\timeDiscrete} \varepsilon_H\right)/\rateTotalOutside{\timeDiscrete}\\
        &= \frac{\left(\infectionConstantMinusOne-\varepsilon_H-(\spectralGap+4\degreeGap+\varepsilon_S)\infectionConstant\right)\infectedDiscrete{\timeDiscrete} + \left(\deimmunizationRate \varepsilon_H - (\spectralGap+4\degreeGap+\varepsilon_S)\infectionConstant\right)\recoveredDiscrete{\timeDiscrete}}{\rateTotalOutside{\timeDiscrete}}\\
        &\geq 0.
    \end{align*}

    The last inequality holds by first choosing $\varepsilon_H < \infectionConstantMinusOne$ and then choosing $\varepsilon_S$ small enough. Then for sufficiently small $\spectralGap$ and $\degreeGap$, both of the summands are positive.

    Note that $\E{T} < \infty$ because in each step $\timeDiscrete \in \N_{<T}$, there is a non-zero probability (independent of \timeDiscrete) to recover a vertex, hence there is always a non-zero probability to recover all vertices within the next \numberOfVertices steps, which stops the process. Therefore, by applying the optional-stopping theorem (\Cref{pre:optionalStopping}), we get $\E{\potentialIMinusR{T}} \geq \E{\potentialIMinusR{0}}$.

    By the law of total expectation, we get that
    \begin{align*}
        \E{\potentialIMinusR{T}} &= \E{\potentialIMinusR{T}} [ \potentialIMinusR{T} \leq 0] \cdot \Pr{\potentialIMinusR{T} \leq 0} + \E{\potentialIMinusR{T}} [ \potentialIMinusR{T} > 0] \cdot \Pr{\potentialIMinusR{T} > 0}\\
        &=\E{\potentialIMinusR{T}} [ \potentialIMinusR{T} \leq 0] \cdot (1-\Pr{\potentialIMinusR{T} > 0}) + \E{\potentialIMinusR{T}} [ \potentialIMinusR{T} > 0] \cdot \Pr{\potentialIMinusR{T} > 0}.
    \end{align*}

    Because of the definition of $T$ and the fact that \potentialIMinusRAlone changes by at most $1+ \varepsilon_H \leq 2$ in one step, we get that $\potentialIMinusR{T} \geq -2$. We also know that $\potentialIMinusR{T} \leq \numberOfVertices$ as $\infectedDiscrete{T} \leq \numberOfVertices$. By definition of \contactProcess, it holds that $\potentialIMinusR{0} \geq 1$. By substituting \E{\potentialIMinusR{T}} in $\E{\potentialIMinusR{T}} \geq \E{\potentialIMinusR{0}}$ and solving for \Pr{\potentialIMinusR{T} > 0}, we get
    \begin{align*}
        \Pr{\potentialIMinusR{T} > 0} &\geq \frac{1 - \E{\potentialIMinusR{T}} [ \potentialIMinusR{T} \leq 0]}{\E{\potentialIMinusR{T}} [ \potentialIMinusR{T} > 0] - \E{\potentialIMinusR{T}} [ \potentialIMinusR{T} \leq 0]}\\
        & \geq \frac{1}{\numberOfVertices+2}.
    \end{align*}

    Now assume $\potentialIMinusR{T} > 0$. By the definition of $T$, it then holds that $\susceptibleDiscrete{T} < (1-\varepsilon_S)\numberOfVertices$. Therefore,
    \begin{align*}
        \infectedDiscrete{T} + \recoveredDiscrete{T} = \numberOfVertices- \susceptibleDiscrete{T} > \varepsilon_S\numberOfVertices.
    \end{align*}
    With $\potentialIMinusR{T} > 0$, we then get $\infectedDiscrete{T} > \varepsilon_H\recoveredDiscrete{T}$, which implies
    \begin{align*}
        \left(1+\varepsilon_H^{-1}\right)\infectedDiscrete{T} &> \varepsilon_S\numberOfVertices.
    \end{align*}
    Choosing $\varepsilon$ accordingly concludes the proof.
\end{proof}

To show that the infection survives long from that point onward, we define a potential function that assigns a real number to each configuration of the process, and we analyze its drift. The potential function is an adjusted version of the Lyapunov function of \textcite{korobeinikov2002lyapunov}. We first define a helper function \lyapunovHelper.

\begin{definition}\label{def:laypunocHelper}
Let $\lyapunovHelper\colon (\R_{> 0})^2 \to \R$ be such that, for all $x,x^* \in \R_{>0}$, we have
\begin{align*}
    \lyapunovHelper[x^*,x] &= x^* \left( \frac{x}{x^*} - \ln \frac{x}{x^*}-1\right).\qedhere
\end{align*}
\end{definition}

Note that the derivative $\frac{\mathrm{d} f(x^*,x)}{\mathrm{d}x}=1- \frac{x^*}{x}$. Hence, for a given $x^* \in \R_{>0}$, the value $x = x^*$ is the only local optimum of $\lyapunovHelper[x^*,x]$, and it is a global minimum. We now define the potential function that we use in the following lemmas.

\begin{definition}\label{def:lyapunovPotential}
Let $G$ be a graph and let $G'$ be a subgraph of $G$ that is an $(\numberOfVertices,(1 \pm \degreeGap)\degree,\spectralGap)$-expander. Let \contactProcess be a SIRS process on $G$ with infection rate $\infectionRate \geq \frac{\infectionConstant}{\degree}$ for a constant $\infectionConstant \in \R_{>1}$ and with deimmunization rate \deimmunizationRate. Consider the projection \contactProcessProjection of \contactProcess onto $G'$. Let $\numberOfVerticesShifted = \left(1+\frac{\deimmunizationRate}{\infectionConstant}\right)\numberOfVertices$.
For all $\timeDiscrete \in \N$, we define $\potentialSIRSClique{\lyapunovParameter}{\timeDiscrete}$ as
\begin{align*}
    \potentialSIRSClique{\lyapunovParameter}{\timeDiscrete} = \lyapunovFunction{\lyapunovParameter}[\susceptibleDiscreteShifted{\timeDiscrete},\infectedDiscrete{\timeDiscrete}] = \lyapunovHelper[\numberOfVerticesShifted,\susceptibleDiscreteShifted{\timeDiscrete}]+ \lyapunovHelper[\infectedEquilibrium,\infectedDiscrete{\timeDiscrete}].
\end{align*}
Further, let $(\filtrationContinuous{\timePoint})_{\timePoint \in \R_{\geq 0}}$ be the natural filtration of \contactProcess. We define for all $\timeDiscrete \in \N$ the \emph{drift} $\driftSIRSClique{\lyapunovParameter}{\timeDiscrete}$ as
\begin{align*}
    \driftSIRSClique{\lyapunovParameter}{\timeDiscrete} &= \E{\potentialSIRSClique{\lyapunovParameter}{\timeDiscrete+1}-\potentialSIRSClique{\lyapunovParameter}{\timeDiscrete}}[\filtrationDiscrete{\timeDiscrete}]. \qedhere
\end{align*}
\end{definition}

The potential \lyapunovFunction{\lyapunovParameter} becomes very large when the infection is close to dying out. We aim to show that the process tends to drift away from that high-potential region when we ignore the impact of the vertices outside of the considered subgraph and that there is a region in which the extra vertices only enlarge that drift.
To calculate the differences of the \lyapunovFunction{\lyapunovParameter} values in the drift, we first have a look at \lyapunovHelper.

\begin{lemma}\label{lem:lyapunovHelper}
    Let $x^* \in \R_{>0}$ and $x \in \R_{>2}$. Then
    \begin{align*}
        &\lyapunovHelper[x^*,x+1] - \lyapunovHelper[x^*,x] \leq 1 - \frac{x^*}{x} + \frac{x^*}{x(x+1)} \textrm{ and}\\
        &\lyapunovHelper[x^*,x-1] - \lyapunovHelper[x^*,x] \leq -\left(1 - \frac{x^*}{x} - \frac{x^*}{x(x-1)}\right).\qedhere
    \end{align*}
\end{lemma}

\begin{proof}

We use that for all $y \in \R_{>1}$, it holds that
\begin{align*}
    \frac{1}{y+1} < \ln(y+1)-\ln(y) <  \frac{1}{y}.
\end{align*}

Together with the definition of \lyapunovHelper, we have
\begin{align*}
    \lyapunovHelper[x^*,x+1] - \lyapunovHelper[x^*,x] &= x^* \left( \frac{x+1}{x^*} - \ln \frac{x+1}{x^*}-1\right) - x^* \left( \frac{x}{x^*} - \ln \frac{x}{x^*}-1\right)\\
    &= 1 - x^* \big(\!\ln(x+1)-\ln x\big)\\
    &\leq 1 - \frac{x^*}{x+1}.
\end{align*}

For the second part, we get
\begin{align*}
    \lyapunovHelper[x^*,x-1] - \lyapunovHelper[x^*,x] &= x^* \left( \frac{x-1}{x^*} - \ln \frac{x-1}{x^*}-1\right) - x^* \left( \frac{x}{x^*} - \ln \frac{x}{x^*}-1\right)\\
    &= -1 + x^* \big(\!\ln x-\ln(x-1)\big)\\
    &\leq -\left(1 - \frac{x^*}{x-1}\right).
\end{align*}

Noting that $\frac{x^*}{x+1} = \frac{x^*}{x}-\frac{x^*}{x(x+1)}$ and $\frac{x^*}{x-1} = \frac{x^*}{x}+\frac{x^*}{x(x-1)}$ concludes the proof.
\end{proof}

To bound the drift, we first show that there is an $\varepsilon \in \R_{> 0}$ such that if there are less than $\varepsilon \numberOfVertices$ infected vertices, the drift is maximized when $\rateInfectionOutside{\timeDiscrete}$ is 0.

\begin{lemma}\label{lem:increasedInfection}
    Let $G$ be a graph, and let $G'$ be a subgraph of $G$ that is an $(\numberOfVertices,(1 \pm \degreeGap)\degree,\spectralGap)$-expander. Let \contactProcess be a SIRS process on $G$ with infection rate~$\infectionRate$ and with constant deimmunization rate \deimmunizationRate. Consider the projection \contactProcessProjection of \contactProcess onto $G'$. Let $\edges{I}{S}$ be the amount of edges between the infected and the susceptible vertices at time $\timeDiscrete$, and let $\rateTotal{\timeDiscrete} = \frac{\infectionConstant}{\degree}\edges{I}{S} + \rateRecover{\timeDiscrete} + \rateSusceptible{\timeDiscrete}$.
    If $\infectionRate \geq \frac{\infectionConstant}{\degree}$ for a constant $\infectionConstant \in \R_{>1}$, then there exists a constant $\varepsilon \in \R_{>0}$ such that, for all $\timeDiscrete \in \N$ and sufficiently large $\numberOfVertices$, if $2 \leq \infectedDiscrete{\timeDiscrete} \leq \varepsilon \numberOfVertices$, then
    \begin{align*}
        \rateTotal{\timeDiscrete} \cdot \driftSIRSClique{\lyapunovParameter}{\timeDiscrete} &\leq\frac{\infectionConstant}{\degree}\edges{I}{S} \cdot \left(\lyapunovFunction{\lyapunovParameter}[\susceptibleDiscreteShifted{\timeDiscrete}-1,\infectedDiscrete{\timeDiscrete}+1]- \lyapunovFunction{\lyapunovParameter}[\susceptibleDiscreteShifted{\timeDiscrete},\infectedDiscrete{\timeDiscrete}]\right)\\
        &\quad + \rateRecover{\timeDiscrete} \cdot \left(\lyapunovFunction{\lyapunovParameter}[\susceptibleDiscreteShifted{\timeDiscrete},\infectedDiscrete{\timeDiscrete}-1]- \lyapunovFunction{\lyapunovParameter}[\susceptibleDiscreteShifted{\timeDiscrete},\infectedDiscrete{\timeDiscrete}]\right)\\
        &\quad + \rateSusceptible{\timeDiscrete} \cdot \left(\lyapunovFunction{\lyapunovParameter}[\susceptibleDiscreteShifted{\timeDiscrete}+1,\infectedDiscrete{\timeDiscrete}]- \lyapunovFunction{\lyapunovParameter}[\susceptibleDiscreteShifted{\timeDiscrete},\infectedDiscrete{\timeDiscrete}]\right).\qedhere
    \end{align*}
\end{lemma}

\begin{proof}
Let $\timeDiscrete \in \N$. For easier notation, we first define
\begin{align*}
    \driftInfection{\timeDiscrete} &= \lyapunovFunction{\lyapunovParameter}[\susceptibleDiscreteShifted{\timeDiscrete}-1,\infectedDiscrete{\timeDiscrete}+1]- \lyapunovFunction{\lyapunovParameter}[\susceptibleDiscreteShifted{\timeDiscrete},\infectedDiscrete{\timeDiscrete}],\\
    \driftRecover{\timeDiscrete} &= \lyapunovFunction{\lyapunovParameter}[\susceptibleDiscreteShifted{\timeDiscrete},\infectedDiscrete{\timeDiscrete}-1]- \lyapunovFunction{\lyapunovParameter}[\susceptibleDiscreteShifted{\timeDiscrete},\infectedDiscrete{\timeDiscrete}]\\
    \text{and } \driftSusceptible{\timeDiscrete} &= \lyapunovFunction{\lyapunovParameter}[\susceptibleDiscreteShifted{\timeDiscrete}+1,\infectedDiscrete{\timeDiscrete}]- \lyapunovFunction{\lyapunovParameter}[\susceptibleDiscreteShifted{\timeDiscrete},\infectedDiscrete{\timeDiscrete}].
\end{align*}

We know that $\rateInfection{\timeDiscrete}= \infectionRate \edges{I}{S} = \frac{\infectionConstant}{\degree} \edges{I}{S} + \rateExtra{\timeDiscrete}$ for some $\rateExtra{\timeDiscrete} \in \R_{\geq 0}$.
By the definition of $\driftSIRSClique{\lyapunovParameter}{\timeDiscrete}$ and the fact that $\rateTotalOutside{\timeDiscrete} = \rateTotal{\timeDiscrete} + \rateInfectionOutside{\timeDiscrete} + \rateExtra{\timeDiscrete}$, we get that
\begin{align*}
    \rateTotal{\timeDiscrete} \cdot \driftSIRSClique{\lyapunovParameter}{\timeDiscrete} &= \rateTotal{\timeDiscrete} \cdot \frac{\rateInfectionOutside{\timeDiscrete} \driftInfection{\timeDiscrete} + \rateInfection{\timeDiscrete} \driftInfection{\timeDiscrete} + \rateRecover{\timeDiscrete} \driftRecover{\timeDiscrete} + \rateSusceptible{\timeDiscrete} \driftSusceptible{\timeDiscrete}}{\rateTotalOutside{\timeDiscrete}}\\
    &= \rateTotal{\timeDiscrete} \cdot \frac{\rateInfectionOutside{\timeDiscrete} \driftInfection{\timeDiscrete} + \rateExtra{\timeDiscrete} \driftInfection{\timeDiscrete} + \frac{\infectionConstant}{\degree} \edges{I}{S} \driftInfection{\timeDiscrete} + \rateRecover{\timeDiscrete} \driftRecover{\timeDiscrete} + \rateSusceptible{\timeDiscrete} \driftSusceptible{\timeDiscrete}}{\rateTotalOutside{\timeDiscrete}}\\
    &= \frac{\infectionConstant}{\degree} \edges{I}{S} \driftInfection{\timeDiscrete} + \rateRecover{\timeDiscrete} \driftRecover{\timeDiscrete} + \rateSusceptible{\timeDiscrete} \driftSusceptible{\timeDiscrete}\\
    &\quad + \frac{\rateInfectionOutside{\timeDiscrete} + \rateExtra{\timeDiscrete}}{\rateTotalOutside{\timeDiscrete}}\left(\rateTotal{\timeDiscrete} \driftInfection{\timeDiscrete} - \frac{\infectionConstant}{\degree} \edges{I}{S} \driftInfection{\timeDiscrete} - \rateRecover{\timeDiscrete} \driftRecover{\timeDiscrete} - \rateSusceptible{\timeDiscrete} \driftSusceptible{\timeDiscrete}\right)
\end{align*}

As $\frac{\rateInfectionOutside{\timeDiscrete}+\rateExtra{\timeDiscrete}}{\rateTotalOutside{\timeDiscrete}}$ is non-negative, to prove the lemma it is sufficient to show that
\begin{align*}
    \rateTotal{\timeDiscrete} \driftInfection{\timeDiscrete} - \frac{\infectionConstant}{\degree} \edges{I}{S} \driftInfection{\timeDiscrete} - \rateRecover{\timeDiscrete} \driftRecover{\timeDiscrete} - \rateSusceptible{\timeDiscrete} \driftSusceptible{\timeDiscrete} \leq 0.
\end{align*}

By \Cref{lem:lyapunovHelper}, we know that for all $x^* \in \R_{>0}$ and $x \in \R_{>2}$ holds
\begin{align*}
    1-\frac{x^*}{x} \leq \lyapunovHelper[x^*,x+1] - \lyapunovHelper[x^*,x] \leq 1 - \frac{x^*}{x+1}.
\end{align*}

Using these bounds, we get that
\begin{align*}
    -\driftRecover{\timeDiscrete} &=  -\left(\lyapunovFunction{\lyapunovParameter}[\susceptibleDiscreteShifted{\timeDiscrete},\infectedDiscrete{\timeDiscrete}-1]- \lyapunovFunction{\lyapunovParameter}[\susceptibleDiscreteShifted{\timeDiscrete},\infectedDiscrete{\timeDiscrete}]\right)\\
    &= - \left(\lyapunovHelper[\infectedEquilibrium,\infectedDiscrete{\timeDiscrete}-1] - \lyapunovHelper[\infectedEquilibrium,\infectedDiscrete{\timeDiscrete}]\right)\\
    &\leq 1-\frac{\infectedEquilibrium}{\infectedDiscrete{\timeDiscrete}} \textrm{, that}\\
    - \driftSusceptible{\timeDiscrete} &= - \left(\lyapunovFunction{\lyapunovParameter}[\susceptibleDiscreteShifted{\timeDiscrete}+1,\infectedDiscrete{\timeDiscrete}]- \lyapunovFunction{\lyapunovParameter}[\susceptibleDiscreteShifted{\timeDiscrete},\infectedDiscrete{\timeDiscrete}]\right)\\
    &= - \left(\lyapunovHelper[\numberOfVerticesShifted,\susceptibleDiscreteShifted{\timeDiscrete}+1] - \lyapunovHelper[\numberOfVerticesShifted,\susceptibleDiscreteShifted{\timeDiscrete}]\right)\\
    &\leq \frac{\numberOfVerticesShifted}{\susceptibleDiscreteShifted{\timeDiscrete}} -1 \textrm{, and that}\\
    \driftInfection{\timeDiscrete} &= \lyapunovFunction{\lyapunovParameter}[\susceptibleDiscreteShifted{\timeDiscrete}-1,\infectedDiscrete{\timeDiscrete}+1]- \lyapunovFunction{\lyapunovParameter}[\susceptibleDiscreteShifted{\timeDiscrete},\infectedDiscrete{\timeDiscrete}]\\
    &= \lyapunovHelper[\numberOfVerticesShifted,\susceptibleDiscreteShifted{\timeDiscrete}-1] - \lyapunovHelper[\numberOfVerticesShifted,\susceptibleDiscreteShifted{\timeDiscrete}] + \lyapunovHelper[\infectedEquilibrium,\infectedDiscrete{\timeDiscrete}+1] - \lyapunovHelper[\infectedEquilibrium,\infectedDiscrete{\timeDiscrete}]\\
    &\leq 1-\frac{\infectedEquilibrium}{\infectedDiscrete{\timeDiscrete}+1} - \left(1-\frac{\numberOfVerticesShifted}{\susceptibleDiscreteShifted{\timeDiscrete}-1}\right)\\
    &= \frac{\numberOfVerticesShifted}{\susceptibleDiscreteShifted{\timeDiscrete}-1} - \frac{\infectedEquilibrium}{\infectedDiscrete{\timeDiscrete}+1}.
\end{align*}

Note that $\numberOfVerticesShifted$ is in $\bigTheta{\numberOfVertices}$ and $\susceptibleDiscreteShifted{\timeDiscrete}$ is bounded from below by $\frac{\deimmunizationRate}{\infectionConstant}\numberOfVertices$, therefore $\frac{\numberOfVerticesShifted}{\susceptibleDiscreteShifted{\timeDiscrete}-1}$ is bounded from above by a constant~$a$. Assume that $\infectedDiscrete{\timeDiscrete} +1 \leq \varepsilon \numberOfVertices$. Let $b=\frac{\deimmunizationRate (1-\infectionConstant)}{(1+\deimmunizationRate)\infectionConstant}$. Note that $b>0$ is constant and $\infectedEquilibrium = b \numberOfVertices$. Using the bounds from above we get
\begin{align*}
    \rateTotal{\timeDiscrete} \driftInfection{\timeDiscrete} - \frac{\infectionConstant}{\degree} \edges{I}{S} \driftInfection{\timeDiscrete} - \rateRecover{\timeDiscrete} \driftRecover{\timeDiscrete} - \rateSusceptible{\timeDiscrete} \driftSusceptible{\timeDiscrete} &= (\rateRecover{\timeDiscrete}+\rateSusceptible{\timeDiscrete}) \driftInfection{\timeDiscrete} - \rateRecover{\timeDiscrete} \driftRecover{\timeDiscrete} - \rateSusceptible{\timeDiscrete} \driftSusceptible{\timeDiscrete}\\
    &\leq (\rateRecover{\timeDiscrete}+\rateSusceptible{\timeDiscrete}) \left(a- \frac{b\numberOfVertices}{\varepsilon \numberOfVertices}\right) + \rateRecover{\timeDiscrete}\left(1-\frac{b \numberOfVertices}{\varepsilon \numberOfVertices}\right) +  \rateSusceptible{\timeDiscrete}(a-1)\\
    &\leq  (\rateRecover{\timeDiscrete}+\rateSusceptible{\timeDiscrete}) \left(2a- \frac{b}{\varepsilon}\right).
\end{align*}

We know that $(\rateRecover{\timeDiscrete}+\rateSusceptible{\timeDiscrete})$ is non-negative, therefore we can choose $\varepsilon$ small enough such that the right-hand side of the previous equation is always at most 0, which concludes the proof.
\end{proof}

We now show that the drift \driftSIRSClique{\lyapunovParameter}{\timeDiscrete} is bounded from above by a negative constant for configurations in which the number of infected vertices is very small but still linear in $\numberOfVertices$.

\begin{lemma}\label{lem:constantDriftSIRS}
    Let $G$ be a graph, and let $G'$ be a subgraph of $G$ that is an $(\numberOfVertices,(1 \pm \degreeGap)\degree,\spectralGap)$-expander. Let $\degree\rightarrow\infty$ and $\spectralGap,\degreeGap\rightarrow0$ as $\numberOfVertices\rightarrow\infty$. Let \contactProcess be a SIRS process on $G$ with infection rate~$\infectionRate$ and with constant deimmunization rate \deimmunizationRate. Consider the projection \contactProcessProjection of \contactProcess onto~$G'$. Let $\timeDiscrete \in \N$ and $\varepsilon_0, \varepsilon \in (0,1)$ be sufficiently small constants. Assume that $\varepsilon_0 \numberOfVertices \geq \infectedDiscrete{\timeDiscrete} \geq \varepsilon \numberOfVertices$.
    If $\infectionRate \geq \frac{\infectionConstant}{\degree}$ for a constant $\infectionConstant \in \R_{>1}$, then there exists a constant $a \in \R_{>0}$ such that $\driftSIRSClique{\lyapunovParameter}{\timeDiscrete} \leq -a$ for sufficiently large \numberOfVertices.
\end{lemma}

\begin{proof}
Let $\edges{I}{S}$ be the amount of edges between the infected and the susceptible vertices at time $\timeDiscrete$, and let $\rateTotal{\timeDiscrete} = \frac{\infectionConstant}{\degree}\edges{I}{S} + \rateRecover{\timeDiscrete} + \rateSusceptible{\timeDiscrete}$.
For this proof, we first use the law of total expectation and \Cref{lem:lyapunovHelper} to get a large formula as an upper bound for $\rateTotal{\timeDiscrete} \driftSIRSClique{\lyapunovParameter}{\timeDiscrete}$. We split this bound into multiple parts and bound each part separately. We show that, with the given conditions, one of the parts is bounded from above by $-m \numberOfVertices$ for a constant $m \in \R_{>0}$, and the other part is in $\smallO{n}$, so it is asymptotically much smaller in absolute values than the other part. We conclude the proof by bounding $\rateTotal{\timeDiscrete}$ and dividing the obtained bound for $\rateTotal{\timeDiscrete} \driftSIRSClique{\lyapunovParameter}{\timeDiscrete}$ by it.

Using \Cref{lem:increasedInfection} and \Cref{lem:lyapunovHelper}, we get
\begin{align*}
    \rateTotal{\timeDiscrete} \cdot \driftSIRSClique{\lyapunovParameter}{\timeDiscrete} &\leq\frac{\infectionConstant}{\degree}\edges{I}{S} \cdot \left(\lyapunovFunction{\lyapunovParameter}[\susceptibleDiscreteShifted{\timeDiscrete}-1,\infectedDiscrete{\timeDiscrete}+1]- \lyapunovFunction{\lyapunovParameter}[\susceptibleDiscreteShifted{\timeDiscrete},\infectedDiscrete{\timeDiscrete}]\right)\\
    &\quad + \rateRecover{\timeDiscrete} \cdot \left(\lyapunovFunction{\lyapunovParameter}[\susceptibleDiscreteShifted{\timeDiscrete},\infectedDiscrete{\timeDiscrete}-1]- \lyapunovFunction{\lyapunovParameter}[\susceptibleDiscreteShifted{\timeDiscrete},\infectedDiscrete{\timeDiscrete}]\right)\\
    &\quad + \rateSusceptible{\timeDiscrete} \cdot \left(\lyapunovFunction{\lyapunovParameter}[\susceptibleDiscreteShifted{\timeDiscrete}+1,\infectedDiscrete{\timeDiscrete}]- \lyapunovFunction{\lyapunovParameter}[\susceptibleDiscreteShifted{\timeDiscrete},\infectedDiscrete{\timeDiscrete}]\right)\\
    &\leq \frac{\infectionConstant}{\degree}\edges{I}{S} \cdot \left( \left(1 - \frac{\infectedEquilibrium}{\infectedDiscrete{\timeDiscrete}} + \frac{\infectedEquilibrium}{\infectedDiscrete{\timeDiscrete}(\infectedDiscrete{\timeDiscrete}+1)}\right) - \left(1 - \frac{\numberOfVerticesShifted}{\susceptibleDiscreteShifted{\timeDiscrete}} - \frac{\numberOfVerticesShifted}{\susceptibleDiscreteShifted{\timeDiscrete}(\susceptibleDiscreteShifted{\timeDiscrete}-1)}\right)\right)\\
    &\quad + \rateRecover{\timeDiscrete} \cdot \left(- \left(1 - \frac{\infectedEquilibrium}{\infectedDiscrete{\timeDiscrete}} - \frac{\infectedEquilibrium}{\infectedDiscrete{\timeDiscrete}(\infectedDiscrete{\timeDiscrete}-1)}\right)\right)\\
    &\quad + \rateSusceptible{\timeDiscrete} \cdot \left(\lyapunovConstant \left(1 - \frac{\numberOfVerticesShifted}{\susceptibleDiscreteShifted{\timeDiscrete}} + \frac{\numberOfVerticesShifted}{\susceptibleDiscreteShifted{\timeDiscrete}(\susceptibleDiscreteShifted{\timeDiscrete}+1)}\right) \right)\\
    &=  \left(1-\frac{\infectedEquilibrium}{\infectedDiscrete{\timeDiscrete}}\right) \left(\frac{\infectionConstant}{\degree}\edges{I}{S} - \rateRecover{\timeDiscrete}\right) + \left(1-\frac{\numberOfVerticesShifted}{\susceptibleDiscreteShifted{\timeDiscrete}}\right)\left(\rateSusceptible{\timeDiscrete} - \frac{\infectionConstant}{\degree}\edges{I}{S}\right)
    \\
    &\quad + \frac{ \frac{\infectionConstant}{\degree}\edges{I}{S} \infectedEquilibrium}{\infectedDiscrete{\timeDiscrete} (\infectedDiscrete{\timeDiscrete}+1)} + \frac{ \frac{\infectionConstant}{\degree}\edges{I}{S} \numberOfVerticesShifted}{\susceptibleDiscreteShifted{\timeDiscrete} (\susceptibleDiscreteShifted{\timeDiscrete}-1)} + \frac{ \rateRecover{\timeDiscrete} \infectedEquilibrium}{\infectedDiscrete{\timeDiscrete} (\infectedDiscrete{\timeDiscrete}-1)} + \frac{ \rateSusceptible{\timeDiscrete} \numberOfVerticesShifted}{\susceptibleDiscreteShifted{\timeDiscrete} (\susceptibleDiscreteShifted{\timeDiscrete}+1)}.
\end{align*}

Note that with the given conditions, all values of \susceptibleDiscreteShifted{\timeDiscrete}, \infectedDiscrete{\timeDiscrete}, \numberOfVerticesShifted, and \infectedEquilibrium are in $\bigTheta{\numberOfVertices}$. All of $\frac{\infectionConstant}{\degree}\edges{I}{S}$, \rateRecover{\timeDiscrete}, and \rateSusceptible{\timeDiscrete} are in $\bigO{\numberOfVertices}$. Therefore, the terms in the second row of the last sum are in $\bigO{1}$, thus we only need to bound the first part.

We know the exact values of $\rateSusceptible{\timeDiscrete}$ and $\rateRecover{\timeDiscrete}$. However, the value of $\frac{\infectionConstant}{\degree}\edges{I}{S}$ depends on which vertices are infected. We use the expander properties of $G$ and  \Cref{pre:expanderbroad} to bound this number. Note that both $\left(1-\frac{\infectedEquilibrium}{\infectedDiscrete{\timeDiscrete}}\right)$ and $\left(1-\frac{\numberOfVerticesShifted}{\susceptibleDiscreteShifted{\timeDiscrete}}\right)$ are negative and lower bounded by some constant. We get for sufficiently large $\numberOfVertices$ that
\begin{align*}
    &\left(1-\frac{\infectedEquilibrium}{\infectedDiscrete{\timeDiscrete}}\right) \left(\frac{\infectionConstant}{\degree}\edges{I}{S} - \rateRecover{\timeDiscrete}\right) + \left(1-\frac{\numberOfVerticesShifted}{\susceptibleDiscreteShifted{\timeDiscrete}}\right)\left(\rateSusceptible{\timeDiscrete} - \frac{\infectionConstant}{\degree}\edges{I}{S}\right)
    \\
    &\leq\left(1-\frac{\infectedEquilibrium}{\infectedDiscrete{\timeDiscrete}}\right) \left(\frac{\infectionConstant}{\numberOfVertices}\infectedDiscrete{\timeDiscrete}\susceptibleDiscrete{\timeDiscrete} - \rateRecover{\timeDiscrete} - 4\degreeGap \frac{\infectionConstant}{\numberOfVertices}\infectedDiscrete{\timeDiscrete}\susceptibleDiscrete{\timeDiscrete} - 2 \infectionConstant \spectralGap \sqrt{\infectedDiscrete{\timeDiscrete}\susceptibleDiscrete{\timeDiscrete}}\right)\\
    &\quad+ \left(1-\frac{\numberOfVerticesShifted}{\susceptibleDiscreteShifted{\timeDiscrete}}\right)\left(\rateSusceptible{\timeDiscrete} - \frac{\infectionConstant}{\numberOfVertices}\infectedDiscrete{\timeDiscrete}\susceptibleDiscrete{\timeDiscrete} - 4\degreeGap \frac{\infectionConstant}{\numberOfVertices}\infectedDiscrete{\timeDiscrete}\susceptibleDiscrete{\timeDiscrete} -2 \infectionConstant \spectralGap \sqrt{\infectedDiscrete{\timeDiscrete}\susceptibleDiscrete{\timeDiscrete}}\right)
    \\
    &\leq\left(1-\frac{\infectedEquilibrium}{\infectedDiscrete{\timeDiscrete}}\right) \left(\frac{\infectionConstant}{\numberOfVertices}\infectedDiscrete{\timeDiscrete}\susceptibleDiscrete{\timeDiscrete} - \rateRecover{\timeDiscrete}\right)+ \left(1-\frac{\numberOfVerticesShifted}{\susceptibleDiscreteShifted{\timeDiscrete}}\right)\left(\rateSusceptible{\timeDiscrete} - \frac{\infectionConstant}{\numberOfVertices}\infectedDiscrete{\timeDiscrete}\susceptibleDiscrete{\timeDiscrete}\right)\\
    &\quad + \left(\frac{\infectedEquilibrium}{\infectedDiscrete{\timeDiscrete}} + \frac{\numberOfVerticesShifted}{\susceptibleDiscreteShifted{\timeDiscrete}}\right) \left(2\infectionConstant \spectralGap \sqrt{\infectedDiscrete{\timeDiscrete}\susceptibleDiscrete{\timeDiscrete}} + 4\degreeGap \frac{\infectionConstant}{\numberOfVertices}\infectedDiscrete{\timeDiscrete}\susceptibleDiscrete{\timeDiscrete}\right).
\end{align*}

Note that $\frac{\infectedEquilibrium}{\infectedDiscrete{\timeDiscrete}} + \frac{\numberOfVerticesShifted}{\susceptibleDiscreteShifted{\timeDiscrete}}$ is in $\bigTheta{1}$, hence the last part of the last sum is in $\bigO{(\spectralGap + \degreeGap) \numberOfVertices}$. As $\spectralGap + \degreeGap$ goes towards~$0$, this is asymptotically smaller then the rest of the drift, which we show now.
\begin{align*}
    & \left(1-\frac{\infectedEquilibrium}{\infectedDiscrete{\timeDiscrete}}\right) \left(\frac{\infectionConstant}{\numberOfVertices} \infectedDiscrete{\timeDiscrete} \susceptibleDiscrete{\timeDiscrete} - \rateRecover{\timeDiscrete}\right) + \left(1-\frac{\numberOfVerticesShifted}{\susceptibleDiscreteShifted{\timeDiscrete}}\right)\left(\rateSusceptible{\timeDiscrete} - \frac{\infectionConstant}{\numberOfVertices} \infectedDiscrete{\timeDiscrete} \susceptibleDiscrete{\timeDiscrete}\right)\\
    &=  \left(1-\frac{\infectedEquilibrium}{\infectedDiscrete{\timeDiscrete}}\right) \left(\frac{\infectionConstant}{\numberOfVertices} \infectedDiscrete{\timeDiscrete} \susceptibleDiscreteShifted{\timeDiscrete} - \deimmunizationRate \infectedDiscrete{\timeDiscrete} - \infectedDiscrete{\timeDiscrete}\right) + \left(1-\frac{\numberOfVerticesShifted}{\susceptibleDiscreteShifted{\timeDiscrete}}\right)\left(\deimmunizationRate \recoveredDiscrete{\timeDiscrete} - \frac{\infectionConstant}{\numberOfVertices} \infectedDiscrete{\timeDiscrete} \susceptibleDiscreteShifted{\timeDiscrete} + \deimmunizationRate \infectedDiscrete{\timeDiscrete}\right)\\
    &=  \left(1-\frac{\infectedEquilibrium}{\infectedDiscrete{\timeDiscrete}}\right) \left(\frac{\infectionConstant}{\numberOfVertices} \infectedDiscrete{\timeDiscrete} \susceptibleDiscreteShifted{\timeDiscrete} - (1+\deimmunizationRate) \infectedDiscrete{\timeDiscrete}\right) + \left(1-\frac{\numberOfVerticesShifted}{\susceptibleDiscreteShifted{\timeDiscrete}}\right)\left(\deimmunizationRate \numberOfVerticesShifted - \deimmunizationRate \susceptibleDiscreteShifted{\timeDiscrete} - \frac{\infectionConstant}{\numberOfVertices} \infectedDiscrete{\timeDiscrete} \susceptibleDiscreteShifted{\timeDiscrete}\right)\\
    &=  \frac{\infectionConstant}{\numberOfVertices} \infectedDiscrete{\timeDiscrete} \susceptibleDiscreteShifted{\timeDiscrete} - (1+\deimmunizationRate) \infectedDiscrete{\timeDiscrete} - \frac{\infectionConstant}{\numberOfVertices} \infectedEquilibrium \susceptibleDiscreteShifted{\timeDiscrete} + (1+\deimmunizationRate) \infectedEquilibrium + \deimmunizationRate \numberOfVerticesShifted- \deimmunizationRate \susceptibleDiscreteShifted{\timeDiscrete} - \frac{\infectionConstant}{\numberOfVertices} \infectedDiscrete{\timeDiscrete} \susceptibleDiscreteShifted{\timeDiscrete} -\deimmunizationRate \frac{\numberOfVerticesShifted{}^2}{\susceptibleDiscreteShifted{\timeDiscrete}} + \deimmunizationRate \numberOfVerticesShifted + \frac{\infectionConstant}{\numberOfVertices} \infectedDiscrete{\timeDiscrete} \numberOfVerticesShifted\\
    &=  (\infectionConstant-1)\infectedDiscrete{\timeDiscrete}- \frac{\deimmunizationRate(\infectionConstant-1)}{1+\deimmunizationRate} \susceptibleDiscreteShifted{\timeDiscrete} + \frac{\deimmunizationRate (\infectionConstant-1)}{\infectionConstant}\numberOfVertices +2 \deimmunizationRate \numberOfVerticesShifted - \deimmunizationRate \susceptibleDiscreteShifted{\timeDiscrete} -\deimmunizationRate \frac{\numberOfVerticesShifted{}^2}{\susceptibleDiscreteShifted{\timeDiscrete}}\\
    &=  \deimmunizationRate  \left(\frac{(\infectionConstant-1)\infectedDiscrete{\timeDiscrete}}{\deimmunizationRate} + \frac{ \infectionConstant-1}{\infectionConstant+\deimmunizationRate}\numberOfVerticesShifted +2  \numberOfVerticesShifted - \frac{\infectionConstant+\deimmunizationRate}{1+\deimmunizationRate} \susceptibleDiscreteShifted{\timeDiscrete} - \frac{\numberOfVerticesShifted{}^2}{\susceptibleDiscreteShifted{\timeDiscrete}}\right).
\end{align*}

We aim to bound this term from above. To this end, we bound $- \frac{(\infectionConstant+\deimmunizationRate)}{1+\deimmunizationRate} \susceptibleDiscreteShifted{\timeDiscrete} - \frac{\numberOfVerticesShifted{}^2}{\susceptibleDiscreteShifted{\timeDiscrete}}$ from above. This term has exactly one maximum for positive $\susceptibleDiscreteShifted{\timeDiscrete}$ which is at $\susceptibleDiscreteShifted{\timeDiscrete} = \numberOfVerticesShifted\sqrt{\frac{1+\deimmunizationRate}{\infectionConstant+\deimmunizationRate}}$. We also bound $\infectedDiscrete{\timeDiscrete} \leq \frac{\varepsilon_0 \infectionConstant}{\infectionConstant+\deimmunizationRate}\numberOfVerticesShifted$ from above. We get
\begin{align*}
    &\deimmunizationRate  \left(\frac{(\infectionConstant-1)\infectedDiscrete{\timeDiscrete}}{\deimmunizationRate} + \frac{ \infectionConstant-1}{\infectionConstant+\deimmunizationRate}\numberOfVerticesShifted +2  \numberOfVerticesShifted - \frac{\infectionConstant+\deimmunizationRate}{1+\deimmunizationRate} \susceptibleDiscreteShifted{\timeDiscrete} - \frac{\numberOfVerticesShifted{}^2}{\susceptibleDiscreteShifted{\timeDiscrete}}\right)\\
    &\leq  \deimmunizationRate \numberOfVerticesShifted \left(\frac{(\infectionConstant-1)\infectionConstant }{\deimmunizationRate (\infectionConstant + \deimmunizationRate)}\varepsilon_0 + \frac{ \infectionConstant-1}{\infectionConstant+\deimmunizationRate} +2  - 2\sqrt{\frac{\infectionConstant+\deimmunizationRate}{1+\deimmunizationRate}} \right).
\end{align*}

The expression in the brackets is a constant, and we aim to show that it is negative for sufficiently small $\varepsilon_0$. We achieve this by showing that the part without the $\varepsilon_0$ is negative and then choosing $\varepsilon_0$ small enough. As both $\deimmunizationRate$ and $\infectionConstant-1$ are positive, we get
\begin{align*}
    &\frac{ \infectionConstant-1}{\infectionConstant+\deimmunizationRate} +2  - 2\sqrt{\frac{\infectionConstant+\deimmunizationRate}{1+\deimmunizationRate}} < 0\\
    \Leftrightarrow\quad & \frac{ \infectionConstant-1}{\infectionConstant+\deimmunizationRate} +2 <2\sqrt{\frac{\infectionConstant+\deimmunizationRate}{1+\deimmunizationRate}}\\
    \Leftrightarrow\quad & \frac{ (\infectionConstant-1)^2}{(\infectionConstant+\deimmunizationRate)^2} +4\frac{ \infectionConstant-1}{\infectionConstant+\deimmunizationRate}+4 <4\frac{\infectionConstant+\deimmunizationRate}{1+\deimmunizationRate}\\
    \Leftrightarrow\quad & \frac{ (\infectionConstant-1)^2}{(\infectionConstant+\deimmunizationRate)^2} +4\frac{ \infectionConstant-1}{\infectionConstant+\deimmunizationRate} <4\frac{\infectionConstant-1}{1+\deimmunizationRate}\\
    \Leftrightarrow\quad & \frac{ (\infectionConstant-1)^2}{(\infectionConstant+\deimmunizationRate)^2} <4\frac{(\infectionConstant-1)^2}{(1+\deimmunizationRate)(\infectionConstant+\deimmunizationRate)}\\
    \Leftrightarrow\quad &1+ \deimmunizationRate < 4(\infectionConstant+\deimmunizationRate).
\end{align*}
The last line holds because $\infectionConstant >1$. Taking everything together, we get that $\rateTotal{\timeDiscrete} \cdot \driftSIRSClique{\lyapunovParameter}{\timeDiscrete}$ is bounded from above by the sum of a constant, by a term that is in $\bigTheta{(\spectralGap + \degreeGap) \numberOfVertices}$, by and $-b \deimmunizationRate \numberOfVerticesShifted$, where~$b$ is a positive constant for sufficiently small~$\varepsilon_0$.

We know that $\rateTotal{\timeDiscrete} = \frac{\infectionConstant}{\degree} \edges{I}{S} + \infectedDiscrete{\timeDiscrete} + \deimmunizationRate \recoveredDiscrete{\timeDiscrete} \leq \infectionConstant \numberOfVertices (1+\degreeGap) + \numberOfVertices + \deimmunizationRate \numberOfVertices = \left(\infectionConstant(1+\degreeGap) + 1 + \deimmunizationRate\right) \numberOfVertices$. As also $\rateTotal{\timeDiscrete} \geq \infectedDiscrete{\timeDiscrete} \geq \varepsilon \numberOfVertices>0$, by dividing the inequality for $\rateTotal{\timeDiscrete} \cdot \driftSIRSClique{\lyapunovParameter}{\timeDiscrete}$ by \rateTotal{\timeDiscrete}, we get that there exists a constant $a \in \R_{>0}$ such that $\driftSIRSClique{\lyapunovParameter}{\timeDiscrete} \leq -a$, concluding the proof.
\end{proof}

We aim to apply the negative-drift theorem (\Cref{pre:NegativeDrift}) to bound the expected survival time of the infection. In \Cref{lem:constantDriftSIRS}, we showed a constant negative drift of the potential in a region of the configuration space. To apply the drift theorem, we first transform the configuration space restrictions into restrictions on the value of the potential. The first lemma shows that if there is at least a constant amount of infected vertices, the potential does not get too large.

\begin{lemma}\label{lem:highEpsilonSIRS}
    Let $G$ be a graph, and let $G'$ be a subgraph of $G$ that is an $(\numberOfVertices,(1 \pm \degreeGap)\degree,\spectralGap)$-expander. Let $\degree\rightarrow\infty$ and $\spectralGap,\degreeGap\rightarrow0$ as $\numberOfVertices\rightarrow\infty$. Let \contactProcess be a SIRS process on $G$ with infection rate~$\infectionRate$ and with constant deimmunization rate \deimmunizationRate. Consider the projection \contactProcessProjection of \contactProcess onto $G'$. Let $\timeDiscrete \in \N$ and $\varepsilon \in (0,1)$ be constants such that $\infectedDiscrete{\timeDiscrete} \geq \varepsilon \numberOfVertices$.
    If $\infectionRate \geq \frac{\infectionConstant}{\degree}$ for a constant $\infectionConstant \in \R_{>1}$, then $\potentialSIRSClique{\lyapunovParameter}{\timeDiscrete} \in \bigO{\numberOfVertices}$.
\end{lemma}

\begin{proof}
    We aim to bound $\potentialSIRSClique{\lyapunovParameter}{\timeDiscrete}$ from above by writing it as a sum and bounding the individual summands. To this end, we first bound the terms that appear in the summands. By the definition of our random variables and the fact that there are only \numberOfVertices vertices, we get
    \begin{align*}
        \max(\susceptibleDiscreteShifted{\timeDiscrete},\infectedDiscrete{\timeDiscrete},\infectedEquilibrium) &\leq \numberOfVerticesShifted,\\
        \min(\susceptibleDiscreteShifted{\timeDiscrete},\infectedDiscrete{\timeDiscrete}) &\geq \min(\varepsilon,\deimmunizationRate/\infectionConstant)\numberOfVertices.
    \end{align*}

    Applying these bounds to \potentialSIRSClique{\lyapunovParameter}{\timeDiscrete} results in
    \begin{align*}
        \potentialSIRSClique{\lyapunovParameter}{\timeDiscrete} &= \lyapunovHelper[\numberOfVerticesShifted,\susceptibleDiscreteShifted{\timeDiscrete}]+ \lyapunovHelper[\infectedEquilibrium,\infectedDiscrete{\timeDiscrete}]\\
        &=  \numberOfVerticesShifted \left(\frac{\susceptibleDiscreteShifted{\timeDiscrete}}{\numberOfVerticesShifted} - \ln \frac{\susceptibleDiscreteShifted{\timeDiscrete}}{\numberOfVerticesShifted} -1 \right) + \infectedEquilibrium \left(\frac{\infectedDiscrete{\timeDiscrete}}{\infectedEquilibrium} - \ln \frac{\infectedDiscrete{\timeDiscrete}}{\infectedEquilibrium} -1 \right)\\
        &\leq  \susceptibleDiscreteShifted{\timeDiscrete} + \numberOfVerticesShifted \ln \frac{\numberOfVerticesShifted}{\susceptibleDiscreteShifted{\timeDiscrete}} +  \infectedDiscrete{\timeDiscrete} + \infectedEquilibrium \ln \frac{\infectedEquilibrium}{\infectedDiscrete{\timeDiscrete}}\\
        &\leq 2 \cdot \left(\numberOfVerticesShifted + \numberOfVerticesShifted \ln \frac{\numberOfVerticesShifted}{\min(\varepsilon,\deimmunizationRate/\infectionConstant)\numberOfVertices}\right).
    \end{align*}

    As $\numberOfVerticesShifted = (1 + \deimmunizationRate/ \infectionConstant) \numberOfVertices$, the calculated upper bound for $\potentialSIRSClique{\lyapunovParameter}{\timeDiscrete}$ is linear in \numberOfVertices. Thus, $\potentialSIRSClique{\lyapunovParameter}{\timeDiscrete} \in \bigO{\numberOfVertices}$.
\end{proof}

The next lemma shows that when the number of vertices becomes small, the potential gets rather large. Together with the previous lemma, this shows that having few infected vertices and having a high drift is more or less the same.

\begin{lemma}\label{lem:lowEpsilonSIRS}
Let $G$ be a graph, and let $G'$ be a subgraph of $G$ that is an $(\numberOfVertices,(1 \pm \degreeGap)\degree,\spectralGap)$-expander. Let $\degree\rightarrow\infty$ and $\spectralGap,\degreeGap\rightarrow0$ as $\numberOfVertices\rightarrow\infty$. Let \contactProcess be a SIRS process on $G$ with infection rate~$\infectionRate$ and with constant deimmunization rate \deimmunizationRate. Consider the projection \contactProcessProjection of \contactProcess onto $G'$. Let $\timeDiscrete \in \N$ and $\varepsilon \in (0,\infectedEquilibrium / \numberOfVertices)$ be constants such that $1 \leq \infectedDiscrete{\timeDiscrete} \leq \varepsilon \numberOfVertices$.
If $\infectionRate \geq \frac{\infectionConstant}{\degree}$ for a constant $\infectionConstant \in \R_{>1}$, then
    \begin{align*}
        \potentialSIRSClique{\lyapunovParameter}{\timeDiscrete} &\geq \infectedEquilibrium \left( \ln \frac{1}{\varepsilon} + \ln \frac{\infectedEquilibrium}{\numberOfVertices} - 1\right). \qedhere
    \end{align*}
\end{lemma}

\begin{proof}
We aim to bound $\potentialSIRSClique{\lyapunovParameter}{\timeDiscrete}$ from below by bounding the $\lyapunovHelper$ values in its definition. Recall that for a given $x^* \in \R_{>0}$, the function $\lyapunovHelper[x^*,x]$ is minimized for $x=x^*$, which is the only local extreme value for $x \in \R_{>0}$. Therefore, we get for all $x,x^* \in \R_{>0}$
\begin{align*}
    \lyapunovHelper[x^*,x] \geq \lyapunovHelper[x^*,x^*] = 0.
\end{align*}

Using $1 \leq \infectedDiscrete{\timeDiscrete} \leq \varepsilon \numberOfVertices$ and that for all $x^* \in \R_{>0}$, the function $\lyapunovHelper[x^*,x]$ decreases monotonically in $x$ while $x<x^*$, we conclude
\begin{align*}
    \potentialSIRSClique{\lyapunovParameter}{\timeDiscrete} &= \lyapunovHelper[\numberOfVerticesShifted,\susceptibleDiscreteShifted{\timeDiscrete}]+ \lyapunovHelper[\infectedEquilibrium,\infectedDiscrete{\timeDiscrete}]\\
    &\geq 0+  \lyapunovHelper[\infectedEquilibrium,\varepsilon \numberOfVertices]\\
    &\geq  \infectedEquilibrium \left( \frac{\varepsilon \numberOfVertices}{\infectedEquilibrium} - \ln \frac{\varepsilon \numberOfVertices}{\infectedEquilibrium} - 1\right) \\
    &\geq  \infectedEquilibrium \left( \ln \frac{1}{\varepsilon} + \ln \frac{\infectedEquilibrium}{\numberOfVertices} - 1\right). \qedhere
\end{align*}
\end{proof}

The next lemma shows that while the process has at least a constant fraction of vertices in the infected state, each potential next step only changes the potential by at most a constant.

\begin{lemma}\label{lem:constamtStepSIRS}
    Let $G$ be a graph, and let $G'$ be a subgraph of $G$ that is an $(\numberOfVertices,(1 \pm \degreeGap)\degree,\spectralGap)$-expander. Let $\degree\rightarrow\infty$ and $\spectralGap,\degreeGap\rightarrow0$ as $\numberOfVertices\rightarrow\infty$. Let \contactProcess be a SIRS process on $G$ with infection rate~$\infectionRate$ and with constant deimmunization rate \deimmunizationRate. Consider the projection \contactProcessProjection of \contactProcess onto~$G'$. Let $\timeDiscrete \in \N$ and $\varepsilon \in (0,\deimmunizationRate/\infectionConstant)$ be constants. Assume that $\infectedDiscrete{\timeDiscrete} \geq \varepsilon \numberOfVertices$. Further, let $\Delta P, \Delta I \in \{-1,0,1\}$.
    If $\infectionRate \geq \frac{\infectionConstant}{\degree}$ for a constant $\infectionConstant \in \R_{>1}$, then for sufficiently large \numberOfVertices, it holds that
    \begin{align*}
        |\lyapunovFunction{\lyapunovParameter}[\susceptibleDiscreteShifted{\timeDiscrete}+\Delta P,\infectedDiscrete{\timeDiscrete}+\Delta I] - \lyapunovFunction{\lyapunovParameter}[\susceptibleDiscreteShifted{\timeDiscrete},\infectedDiscrete{\timeDiscrete}]| &\leq 2\left(1+2 (1+ \deimmunizationRate/\infectionConstant)\varepsilon^{-1}\right). \qedhere
    \end{align*}
\end{lemma}

\begin{proof}
We aim to use the triangle inequality to bound the absolute change in the $\lyapunovFunction{\lyapunovParameter}$-values from above by the sum of the absolute changes in the $\lyapunovHelper$-values. We use that for all $x \in \R_{>1}$ holds that
\begin{align*}
    \frac{1}{x+1} < \ln\left(\frac{x+1}{x}\right) <  \frac{1}{x}.
\end{align*}
Further, for all $x, x^* \in \R_{>2}$ and $\Delta x \in \{-1,0,1\}$ holds that
\begin{align*}
    |\lyapunovHelper[x^*,x+\Delta x]- \lyapunovHelper[x^*,x]| &= \left|x^* \left( \frac{x+\Delta x}{x^*} - \ln \frac{x+\Delta x}{x^*}-1\right) - x^* \left( \frac{x}{x^*} - \ln \frac{x}{x^*}-1\right)\right|\\
    &= \left| \Delta x - x^* \ln\left(\frac{x + \Delta x}{x} \right)\right|\\
    &\leq |\Delta x| + \left| x^* \ln\left(\frac{x + \Delta x}{x} \right)\right|\\
    &\leq 1 + \frac{x^*}{x-1}.
\end{align*}

We apply this inequality to bound the absolute change in potential from above. Note that for sufficiently large \numberOfVertices, it holds that $\min(\susceptibleDiscreteShifted{\timeDiscrete}-1,\infectedDiscrete{\timeDiscrete}-1) \geq \varepsilon \numberOfVertices/2$. We conclude
\begin{align*}
    &|\lyapunovFunction{\lyapunovParameter}[\susceptibleDiscreteShifted{\timeDiscrete}+\Delta P,\infectedDiscrete{\timeDiscrete}+\Delta I] - \lyapunovFunction{\lyapunovParameter}[\susceptibleDiscreteShifted{\timeDiscrete},\infectedDiscrete{\timeDiscrete}]|\\
    &= \left| \lyapunovHelper[\numberOfVerticesShifted,\susceptibleDiscreteShifted{\timeDiscrete}+\Delta P] + \lyapunovHelper[\infectedEquilibrium, \infectedDiscrete{\timeDiscrete}+\Delta I]- \lyapunovHelper[\numberOfVerticesShifted,\susceptibleDiscreteShifted{\timeDiscrete}] - \lyapunovHelper[\infectedEquilibrium, \infectedDiscrete{\timeDiscrete}]\right|\\
    &\leq \left| \lyapunovHelper[\numberOfVerticesShifted,\susceptibleDiscreteShifted{\timeDiscrete}+\Delta P] -
    \lyapunovHelper[\numberOfVerticesShifted,\susceptibleDiscreteShifted{\timeDiscrete}] \right|
    + \left| \lyapunovHelper[\infectedEquilibrium, \infectedDiscrete{\timeDiscrete}+\Delta I] - \lyapunovHelper[\infectedEquilibrium, \infectedDiscrete{\timeDiscrete}] \right|\\
    &\leq \left(1+ \frac{\numberOfVerticesShifted}{\susceptibleDiscreteShifted{\timeDiscrete}-1}\right) + \left(1 + \frac{\infectedEquilibrium}{\infectedDiscrete{\timeDiscrete}-1}\right)\\
    &\leq 2(1+ \frac{\numberOfVerticesShifted}{\varepsilon \numberOfVertices/2})\\
    &\leq 2(1+2 (1+ \deimmunizationRate/\infectionConstant)\varepsilon^{-1}).\qedhere
\end{align*}
\end{proof}

We now have the tools to apply the negative-drift theorem (\Cref{pre:NegativeDrift}) to bound the survival time of the infection.

\begin{lemma}\label{lem:longSurvivalCliqueSIRS}
    Let $G$ be a graph, and let $G'$ be a subgraph of $G$ that is an $(\numberOfVertices,(1 \pm \degreeGap)\degree,\spectralGap)$-expander. Let $\degree\rightarrow\infty$ and $\spectralGap,\degreeGap\rightarrow0$ as $\numberOfVertices\rightarrow\infty$. Let \contactProcess be a SIRS process on $G$ with infection rate~$\infectionRate$ and with constant deimmunization rate \deimmunizationRate. Consider the projection~\contactProcessProjection of~\contactProcess onto~$G'$. Let $\varepsilon_0 \in (0,1)$ be a constant and let $E_{\varepsilon_0}$ be the event that there exists a $\timeDiscrete \in \N$ such that $\infectedDiscrete{\timeDiscrete} \geq \varepsilon_0 \numberOfVertices$. Let $T$ be the first time after $\timeContinuous{\timeDiscrete}$ with $\infectedDiscrete{\timeDiscrete}=0$.
    If $\infectionRate \geq \frac{\infectionConstant}{\degree}$ for a constant $\infectionConstant \in \R_{>1}$, then $\E{T}[E_{\varepsilon_0}] = 2^{ \bigOmega{\numberOfVertices}}$.
\end{lemma}

\begin{proof}
We assume that $E_{\varepsilon_0}$ occurs. Let $(\filtrationContinuous{\timePoint})_{\timePoint \in \R_{\geq 0}}$ be the natural filtration of \contactProcess, and let $\timeDiscrete \in \N$ be such that $\infectedDiscrete{\timeDiscrete} \geq \varepsilon_0 \numberOfVertices$. We aim to apply the negative-drift theorem (\Cref{pre:NegativeDrift}) to get the desired bound. To this end, we define a stopping time that is dominated by the number of steps until $T$, and we use the previous lemmas to show that all of the conditions for the drift theorem are satisfied. Note that we shift the time to start at $\timeDiscrete$ instead of $0$. We then translate the bound on the number of steps into a bound on the survival time.

As $\infectedDiscrete{\timeDiscrete} \geq \varepsilon_0 \numberOfVertices$, by \Cref{lem:highEpsilonSIRS}, there exists a constant $a_0 \in \R_{>0}$ such that $\potentialSIRSClique{\lyapunovParameter}{\timeDiscrete} \leq a_0 \numberOfVertices$. Let $\varepsilon_c$ be the minimum of the constants from \Cref{lem:increasedInfection,lem:constantDriftSIRS}. By the contraposition of \Cref{lem:highEpsilonSIRS}, there exists a constant $a_1 \in \R_{>0}$ such that $\potentialSIRSClique{\lyapunovParameter}{\timeDiscrete} \geq a_1 \numberOfVertices$ implies that $\infectedDiscrete{\timeDiscrete} \leq \varepsilon_c \numberOfVertices$. We define $a = \max(a_0,a_1)$ and $T_1 = \inf\{i \in \N_{\geq t} \mid \potentialSIRSClique{\lyapunovParameter}{i} > 2a\numberOfVertices \}$.

We first show that for all $i \in \N$ with $\timeDiscrete \leq i < T_1$ holds that $\infectedDiscrete{i}$ is large enough such that \Cref{lem:constamtStepSIRS} is applicable. Let $\varepsilon_1 \in (0,\infectedEquilibrium/\numberOfVertices)$ be a constant low enough such that $ \frac{\infectedEquilibrium}{\numberOfVertices} \left( \ln \frac{1}{\varepsilon_1} + \ln \frac{\infectedEquilibrium}{\numberOfVertices} - 1\right) > 2a$. Such an $\varepsilon_1$ exists, as $\frac{\infectedEquilibrium}{\numberOfVertices}$ and $a$ are positive constants. Then by the contraposition of \Cref{lem:lowEpsilonSIRS}, for all $i \in \N$, it follows that $\potentialSIRSClique{\lyapunovParameter}{i} \leq 2a \numberOfVertices$ implies that $\infectedDiscrete{i} \geq \varepsilon_1 \numberOfVertices$.

To show that condition 2 of \Cref{pre:NegativeDrift} is satisfied, let $s = 2(1+2 (1+ \deimmunizationRate/\infectionConstant)\varepsilon_1^{-1})$. For all $i \in \N$ with $\timeDiscrete \leq i < T_1$ holds $\potentialSIRSClique{\lyapunovParameter}{i} \leq 2a \numberOfVertices$ and therefore $\infectedDiscrete{i} \geq \varepsilon_1 \numberOfVertices$. Hence, by \Cref{lem:constamtStepSIRS}, for all $i \in \N_{\geq \timeDiscrete}$ holds that $|\potentialSIRSClique{\lyapunovParameter}{i+1} - \potentialSIRSClique{\lyapunovParameter}{i}| \cdot \indicator{i<T_1}\leq s \cdot \indicator{i<T_1}$. Thus, for all $i \in \N_{\geq \timeDiscrete}$ and $j \in \R_{>0}$ holds that $\Pr{|\potentialSIRSClique{\lyapunovParameter}{i+1} - \potentialSIRSClique{\lyapunovParameter}{i}|\geq j}[\filtrationDiscrete{i}] \cdot \indicator{i<T_1} \leq \frac{2^s}{2^j} \cdot \indicator{i<T_1}$. Note that $2^s$ is a constant.

We now show that condition 1 is satisfied as well. We already showed that for all $i \in \N$ with $a \numberOfVertices< \potentialSIRSClique{\lyapunovParameter}{i}< 2a \numberOfVertices$ holds $\varepsilon_1 \numberOfVertices \leq \infectedDiscrete{\timeDiscrete} \leq \varepsilon_c\numberOfVertices$. Hence, the conditions for \Cref{lem:constantDriftSIRS} are satisfied, and we get that there exists a constant $r \in \R_{>0}$ such that for all $i \in \N$ holds that $\E{\potentialSIRSClique{\lyapunovParameter}{i+1} - \potentialSIRSClique{\lyapunovParameter}{i}}[\filtrationDiscrete{i}] \cdot \indicator{a \numberOfVertices< \potentialSIRSClique{\lyapunovParameter}{i}< 2a \numberOfVertices} \leq -r \cdot \indicator{a \numberOfVertices< \potentialSIRSClique{\lyapunovParameter}{i}< 2a \numberOfVertices}$.

Now all of the conditions of \Cref{pre:NegativeDrift} are satisfied, and we get that there exists a constant $c^* \in \R_{>0}$ such that
\begin{align*}
    \Pr{T_1 - \timeDiscrete \leq 2^{c^* a \numberOfVertices / 2^s}}[\filtrationDiscrete{\timeDiscrete}] \cdot \indicator{\potentialSIRSClique{\lyapunovParameter}{\timeDiscrete} \leq a \numberOfVertices} = 2^{-\bigOmega{a \numberOfVertices / 2^s}} \cdot \indicator{\potentialSIRSClique{\lyapunovParameter}{\timeDiscrete} \leq a \numberOfVertices}.
\end{align*}

Note that this probability goes towards~$0$ as \numberOfVertices goes towards infinity. Hence, $\E{T_1}[\filtrationDiscrete{\timeDiscrete}] \cdot \indicator{\potentialSIRSClique{\lyapunovParameter}{\timeDiscrete} \leq a \numberOfVertices}= 2^{ \bigOmega{\numberOfVertices}} \cdot \indicator{\potentialSIRSClique{\lyapunovParameter}{\timeDiscrete} \leq a \numberOfVertices}$. Remember that $\infectedDiscrete{\timeDiscrete} \geq \varepsilon_0 \numberOfVertices$ implies $\potentialSIRSClique{\lyapunovParameter}{\timeDiscrete} \leq a \numberOfVertices$. We therefore get $\E{T_1}[\filtrationDiscrete{\timeDiscrete}] \cdot \indicator{\infectedDiscrete{\timeDiscrete} \geq \varepsilon_0 \numberOfVertices}= 2^{ \bigOmega{\numberOfVertices}} \cdot \indicator{\infectedDiscrete{\timeDiscrete} \geq \varepsilon_0 \numberOfVertices}$.

We showed that for all $i \in \N$ with $\timeDiscrete \leq i < T_1$ holds that $\infectedDiscrete{i} \geq \varepsilon_1 \numberOfVertices >0$. Thus,~$T$ dominates~$\timeContinuous{T_1}$. Note that clocks in~$\contactProcess$ trigger at an arbitrarily high rate, as we do not have an upper bound on~$\rateTotalOutside{\timeDiscrete}$. However, the amounts of recovery triggers, infection triggers, and deimmunization triggers that occur until $\timeContinuous{T_1}$ differ by at most $\numberOfVertices$, pairwise by type, and each of them also has an exponential expectation. As we only consider $\numberOfVertices$ recovery clocks, they trigger at a rate of at most~$\numberOfVertices$, and the expected time between each trigger is at least $\frac{1}{\numberOfVertices}$. By Wald's equation (\Cref{pre:wald}), we get that
\begin{align*}
    \E{T}[\filtrationContinuous{0}] \cdot \indicator{E_{\varepsilon_0}} &\geq \E{\timeContinuous{T_1}}[\filtrationContinuous{0}] \cdot \indicator{E_{\varepsilon_0}}\\
    &\geq \frac{1}{\numberOfVertices}2^{ \bigOmega{\numberOfVertices}} \cdot \indicator{E_{\varepsilon_0}}. \qedhere
\end{align*}
\end{proof}

We now prove our main result.

\SIRSClique

\begin{proof}
For all $\varepsilon \in (0,1)$, let~$E_\varepsilon$ be the event that there exists a $\timeDiscrete \in \N$ such that $\infectedDiscrete{\timeDiscrete} \geq \varepsilon \numberOfVertices$. By \Cref{lem:farFromEdgeSIRS}, there exists an $\varepsilon \in \R_{>0}$ such that for sufficiently large \numberOfVertices holds that $\Pr{E_\varepsilon} \geq \frac{1}{\numberOfVertices +2}$. By \Cref{lem:longSurvivalCliqueSIRS}, it holds that $\E{T}[E_\varepsilon]=2^{ \bigOmega{\numberOfVertices}}$. Using the law of total expectation, we get
\begin{align*}
    \E{T} &= \Pr{E_\varepsilon} \E{T}[E_\varepsilon] + \Pr{\overline{E_\varepsilon}}[][\big] \E{T}[\overline{E_\varepsilon}][\big]\\
    &\geq \Pr{E_\varepsilon} \E{T}[E_\varepsilon]\\
    &\geq \frac{1}{\numberOfVertices +2} 2^{ \bigOmega{\numberOfVertices}}\\
    &= 2^{ \bigOmega{\numberOfVertices}}. \qedhere
\end{align*}
\end{proof}

\section{Special graph classes}
\label{sec:graph_classes}
We present the implications of \Cref{thm:cliqueSIRS} on special graph classes. We focus our attention to two random graph models, namely \erdosGraphs and hyperbolic random graphs.

\subsection{\erdosGraphs}

In order to apply \cref{thm:cliqueSIRS} to \erdosGraphs, we make use of the following result.

\begin{theorem}[{\cite[Theorem~$1.2$]{coja2007laplacian}}]\label{pre:erdosExpansion}
Let $G \sim G_{\numberOfVertices,p}$ be an \erdosGraph with $(\numberOfVertices-1)p \geq c_1 \ln(\numberOfVertices)$ for a sufficiently large constant $c_1 \in \R_{>0}$. Then asymptotically almost surely, for the spectral expansion $\spectralGap$ of the Laplacian of $G$ holds $\spectralGap \in \bigO{(p(\numberOfVertices-1))^{-1/2})}$.
\end{theorem}

By Chernoff bounds, it holds that the vertex degrees in \erdosGraphs are tightly distributed around the average degree $\degree$ if $\degree \in \smallOmega{\ln \numberOfVertices}$. Therefore, \erdosGraphs satisfy with high probability our definition of an $(\numberOfVertices,(1 \pm \degreeGap)\degree,\spectralGap)$-expander. Combining this with \Cref{thm:cliqueSIRS}, we obtain the following corollary.

\ERGraphs

\subsection{Hyperbolic random graphs}

For the formal definition of a hyperbolic random graph, we refer the reader to the article by \textcite{KPKVB10}.
The two key properties we require for our main result to be applicable on hyperbolic random graphs are the following.

\begin{theorem}[{\cite[Theorem~$1$]{friedrich2018diameter}}]\label{pre:hyperbolicDiameter}
Let $G$ be a hyperbolic random graph with $\numberOfVertices$ vertices that follows a power-law degree distribution with exponent $\hyperbolicExponent \in (2,3)$. Then the diameter of the giant component of $G$ is \bigO{(\log n)^{2/(3-\hyperbolicExponent)}} with probability $1- \bigO{n^{-3/2}}$.
\end{theorem}

\begin{theorem}[{\cite{friedrich2015cliques}}]\label{pre:hyperbolicClique}
Let $G$ be a hyperbolic random graph with $\numberOfVertices$ vertices that follows a power-law degree distribution with exponent $\hyperbolicExponent \in (2,3)$. Then the size of the largest clique of $G$ is in \bigTheta{n^{(3-\hyperbolicExponent)/2}} with high probability.
\end{theorem}

We first use the poly-logarithmic diameter to show that the infection reaches the largest clique with a sufficient probability when the process starts with at least one infected vertex.

\begin{lemma}\label{lem:hyperbolicPath}
Let $G$ be a hyperbolic random graph with $\numberOfVertices$ vertices that follows a power-law degree distribution with exponent $\hyperbolicExponent \in (2,3)$, and let $\contactProcess$ be an SIRS process on $G$ with infection rate~$\infectionRate$ and with constant deimmunization rate \deimmunizationRate. Further, let~$\contactProcess$ start with at least one infected vertex in the giant component and no recovered vertices in the giant component.
If $\infectionRate \geq \infectionConstant \numberOfVertices^{(\hyperbolicExponent-3)/2}$ for a constant $\infectionConstant \in \R_{>0}$, then the probability that the infection reaches a state in which a vertex in the largest clique is infected is at least $\exp\!\big(\!-\!(\ln n)^{3/(3-\hyperbolicExponent)}\big)$ for sufficiently large~$\numberOfVertices$.
\end{lemma}
\begin{proof}
Let $v$ be a vertex that starts infected, and let $d$ be the shortest distance from $v$ to any vertex of the largest clique. Note that $d$ is bounded from above by the diameter of the giant component. Therefore, by \Cref{pre:hyperbolicDiameter}, there exists a constant $a \in \R_{>0}$ such that for sufficiently large $\numberOfVertices$ with a probability of at least $\frac{1}{2}$, it holds that $d\leq a (\ln \numberOfVertices)^{2/(3-\hyperbolicExponent)}$.

For all $i\in \N$, let~$E_i$ be the event that $\contactProcess$ reaches a state with an infected vertex that has a distance of $i$ to the largest clique. Consider for all $i \in \N_{<d}$ the probability $\Pr{E_i}[E_{i+1}]$. Each vertex with a distance of $i+1$ to the largest clique has a neighbor that has a distance of $i$ to the clique. With a probability of $\frac{\infectionRate}{1+\infectionRate}$, an infected vertex infects a specific neighbor before recovering. Therefore, $\Pr{E_i}[E_{i+1}] \geq \frac{\infectionRate}{1+\infectionRate} \geq \frac{\infectionConstant}{2} \numberOfVertices^{(\hyperbolicExponent-3)/(2)}$ for sufficiently large $\numberOfVertices$.

With a probability of at least $\frac{1}{2}$, it holds that $d\leq a (\ln \numberOfVertices)^{2/(3-\hyperbolicExponent)}$. This yields for sufficiently large~$\numberOfVertices$ that
\begin{align*}
    \Pr{E_0} &= \prod_{i=0}^{d-1}{\Pr{E_i}[E_{i+1}]}\\
    &\geq \prod_{i=0}^{d-1}{\frac{\infectionConstant}{2} \numberOfVertices^{\frac{\hyperbolicExponent-3}{2}}}\\
    &= \left(\frac{\infectionConstant}{2} \numberOfVertices^{\frac{\hyperbolicExponent-3}{2}}\right)^d\\
    &\geq \left(\frac{\infectionConstant}{2} \numberOfVertices^{\frac{\hyperbolicExponent-3}{2}}\right)^{a (\ln \numberOfVertices)^{\frac{2}{3-\hyperbolicExponent}}}\\
    &=\eulerE^{\frac{\hyperbolicExponent-3}{2}a(\ln n)^{\frac{5-\hyperbolicExponent}{3-\hyperbolicExponent}}+\ln(c/2)}\\
    &\geq \eulerE^{-(\ln n)^{\frac{3}{3-\hyperbolicExponent}}}.\qedhere
\end{align*}
\end{proof}

When the infection reaches the largest clique of a hyperbolic random graph, \Cref{thm:cliqueSIRS} yields an exponential expected survival time for a sufficiently large infection rate.

\HRGGraphs
\begin{proof}
Let $k$ be the size of the largest clique of $G$. By \Cref{pre:hyperbolicClique}, there exists a constant $a \in \R_{>0}$ such that with high probability it holds that $k \geq a n^{(3-\hyperbolicExponent)/2}$. Let $\infectionConstant = a^{-1}+1$ such that with high probability it holds that $\infectionRate \geq \frac{1+a}{k}$. Let $E$ be the event that there exists a configuration in which a vertex in the largest clique of $G$ is infected. By \Cref{lem:hyperbolicPath}, it holds that $\Pr{E} \geq \exp\!\big(\!-(\ln n)^{3/(3-\hyperbolicExponent)}\big)$ for sufficiently large $\numberOfVertices$. Note that a clique with $k$ vertices is an $(k,(1\pm k^{-1})k,(k-1)^{-1})$-expander. Hence, by \Cref{thm:cliqueSIRS}, it holds that $\E{T}[E]= 2^{\bigOmega{k}}$, as the infection survives that long on the clique alone after its first vertex gets infected.

By the law of total expectation and that with high probability $k \geq a n^{(3-\hyperbolicExponent)/2}$, we conclude
\begin{align*}
    \E{T} &\geq \Pr{E} \cdot \E{T}[E]\\
    &\geq \eulerE^{-(\ln n)^{\frac{3}{3-\hyperbolicExponent}}} \cdot 2^{\bigOmega{n^{(3-\hyperbolicExponent)/2}}}\\
    &= 2^{\bigOmega{n^{(3-\hyperbolicExponent)/2}}}.\qedhere
\end{align*}
\end{proof}

\section*{Acknowledgments}
We would like to thank Silvio Ferreira for bringing \cite{ferreira2016collective} to our attention and for insightful comments that helped us improve the expected survival time of \Cref{lem:starSurvival}.

Andreas Göbel was funded by the project PAGES (project No. 467516565) of the German Research Foundation (DFG).
This project has received funding from the European Union's Horizon 2020 research and innovation program under the Marie Skłodowska-Curie grant agreement No. 945298-ParisRegionFP.
This research was partially funded by the HPI Research School on Data Science and Engineering.

\printbibliography

@article{berger2005spread,
  title={On the spread of viruses on the internet},
  author={Berger, Noam and Borgs, Christian and Chayes, Jennifer T. and Saberi, Amin},
  year={2005},
  journal={Symposium on Discrete Algorithms (SODA)},
  publisher={Society for Industrial and Applied Mathematics},
  pages = {301–310},
  url={https://dl.acm.org/doi/10.5555/1070432.1070475},
}

@inproceedings{ganesh2005effect,
  title={The effect of network topology on the spread of epidemics},
  author={Ganesh, Ayalvadi and Massouli{\'e}, Laurent and Towsley, Don},
  booktitle={International Conference on Computer Communications (INFOCOM)},
  pages={1455--1466},
  year={2005},
  doi={10.1109/INFCOM.2005.1498374},
}

@article{korobeinikov2002lyapunov,
  title={Lyapunov functions and global stability for SIR, SIRS, and SIS epidemiological models},
  author={Korobeinikov, Andrei and Wake, Graeme C.},
  journal={Applied Mathematics Letters},
  volume={15},
  number={8},
  pages={955--960},
  year={2002},
  doi={10.1016/S0893-9659(02)00069-1},
}

@article{oliveto2011simplified,
  title={Simplified drift analysis for proving lower bounds in evolutionary computation},
  author={Oliveto, Pietro S. and Witt, Carsten},
  journal={Algorithmica},
  volume={59},
  number={3},
  pages={369--386},
  year={2011},
  doi={10.1007/s00453-010-9387-z},
}

@article{OlivetoW12NegativeDriftErratum,
  author    = {Pietro S. Oliveto and Carsten Witt},
  title     = {Erratum: Simplified drift analysis for proving lower bounds in evolutionary computation},
  journal   = {CoRR},
  volume    = {abs/1211.7184},
  year      = {2012},
  url       = {http://arxiv.org/abs/1211.7184},
}

@book{mitzenmacher2017probability,
  title={Probability and Computing: Randomization and Probabilistic Techniques in Algorithms and Data Analysis},
  author={Mitzenmacher, Michael and Upfal, Eli},
  year={2017},
  publisher={Cambridge university press},
  isbn={978-1-107-15488-9},
  edition={2}
}

@book{feller1957introduction,
  title={An Introduction to Probability Theory and its Applications},
  edition={3},
  volume={1},
  author={Feller, William},
  year={1968},
  publisher={John Wiley \& Sons},
  ISBN={978-0-471-25708-0}
}

@article{saif2019epidemic,
  title={Epidemic threshold for the SIRS model on the networks},
  author={Saif, M. Ali},
  journal={Physica A: Statistical Mechanics and its Applications},
  volume={535},
  pages={122251.1--122251.7},
  year={2019},
  doi={10.1016/j.physa.2019.122251},
}

@article{NamNS22SISinfinite,
  title={Critical value asymptotics for the contact process on random graphs},
  author={Danny Nam and Oanh Nguyen and Allan Sly},
  journal={Transactions of the American Mathematical Society},
  volume={375},
  number={12},
  year={2022},
  doi={10.1090/tran/8399},
}

@article{BorgsCGS10Antidote,
    author = {Christian Borgs and Jennifer Chayes and Ayalvadi Ganesh and Amin Saberi},
    title = {How to distribute antidote to control epidemics},
    journal = {Random Structures \& Algorithms},
    volume = {37},
    number = {2},
    pages = {204-222},
    year = {2010},
    doi = {10.1002/rsa.20315},
}

@article{Pastor-SatorrasCVMV15Survey,
  title = {Epidemic processes in complex networks},
  author = {Romualdo Pastor-Satorras and Claudio Castellano and Piet Van Mieghem and Alessandro Vespignani},
  journal = {Reviews of Modern Physics},
  volume = {87},
  number = {3},
  pages = {925--979},
  year = {2015},
  doi = {10.1103/RevModPhys.87.925},
}

@article{Hajek82HittingTime,
    author    = {Bruce Hajek},
    title     = {Hitting-time and occupation-time bounds implied by drift analysis with applications},
    journal   = {Advances in Applied Probability},
    volume    = {14},
    number    = {3},
    pages     = {502--525},
    year      = {1982},
    doi       = {10.2307/1426671},
}

@article{DoerrK22WaldsEquation,
  author    = {Carola Doerr and Martin S. Krejca},
  title     = {Run time analysis for random local search on generalized majority functions},
  journal   = {IEEE Transactions on Evolutionary Computation},
  year      = {2022},
  doi       = {10.1109/TEVC.2022.3216349},
  note      = {In press},
}

@article{Harris74,
    author = {T. E. Harris},
    title = {Contact interactions on a lattice},
    volume = {2},
    journal = {The Annals of Probability},
    number = {6},
    pages = {969 -- 988},
    year = {1974},
    doi = {10.1214/aop/1176996493},
}

@article{Pemantle1992TheCP,
  title={The Contact Process on Trees},
  author={Robin Pemantle},
  journal={Annals of Probability},
  year={1992},
  volume={20},
  pages={2089-2116},
  doi={10.1214/AOP/1176989541},
}

@article{10.2307/2959578,
 author = {Thomas M. Liggett},
 journal = {The Annals of Probability},
 number = {4},
 pages = {1675--1710},
 title = {Multiple transition points for the contact process on the binary tree},
 volume = {24},
 year = {1996},
 doi = {10.1214/aop/1041903202},
}

@article{Stacey96,
    author = {Alan M. Stacey},
    title = {The existence of an intermediate phase for the contact process on trees},
    volume = {24},
    journal = {The Annals of Probability},
    number = {4},
    publisher = {Institute of Mathematical Statistics},
    pages = {1711 -- 1726},
    year = {1996},
    doi = {10.1214/aop/1041903203},
}

@article{huang2020contact,
  title={The Contact Process on Random Graphs and Galton Watson Trees},
  author={Huang, Xiangying and Durrett, Rick},
  journal={Latin American Journal of Probability and Mathematical Statistics},
  volume={17},
  pages={159--182},
  year={2020},
  doi={10.30757/alea.v17-07},
}

@article{bhamidi2021survival,
  title={Survival and extinction of epidemics on random graphs with general degree},
  author={Bhamidi, Shankar and Nam, Dannz and Nguyen, Oanh and Sly, Allan},
  journal={The Annals of Probability},
  volume={49},
  number={1},
  pages={244--286},
  year={2021},
  doi={10.1214/20-AOP1451},
}

@article{Wang_2017,
	doi = {10.1088/1742-5468/aa58a6},
	year = {2017},
	publisher = {{IOP} Publishing},
	volume = {2017},
	number = {2},
	pages = {1--26},
	author = {Yi Wang and Jinde Cao and Ahmed Alsaedi and Tasawar Hayat},
	title = {The spreading dynamics of sexually transmitted diseases with birth and death on heterogeneous networks},
	journal = {Journal of Statistical Mechanics: Theory and Experiment},
}

@article{PhysRevLett.86.2909,
  title = {Small world effect in an epidemiological model},
  author = {Kuperman, Marcelo and Abramson, Guillermo},
  journal = {Physical Review Letters},
  volume = {86},
  issue = {13},
  pages = {2909--2912},
  year = {2001},
  doi = {10.1103/PhysRevLett.86.2909},
}

@article{Bancal10,
	author = {Bancal, Jean-Daniel and Pastor-Satorras, Romualdo},
	title = {Steady-state dynamics of the forest fire model on complex networks},
	journal = {The European Physical Journal B},
	year = {2010},
	volume = {76},
	number = {1},
	pages = {109-121},
    doi = {10.1140/epjb/e2010-00165-7},
}

@article{KPKVB10,
  title = {Hyperbolic geometry of complex networks},
  author = {Krioukov, Dmitri and Papadopoulos, Fragkiskos and Kitsak, Maksim and Vahdat, Amin and Bogu\~n\'a, Mari\'an},
  journal = {Physical Review E},
  volume = {82},
  issue = {3},
  pages = {036106-1--036106-18},
  numpages = {18},
  year = {2010},
  publisher = {American Physical Society},
  doi = {10.1103/PhysRevE.82.036106},
}

@article{ferreira2016collective,
  title={Collective versus hub activation of epidemic phases on networks},
  author={Ferreira, Silvio C. and Sander, Renan S. and Pastor-Satorras, Romualdo},
  journal={Physical Review E},
  volume={93},
  number={3},
  pages={032314},
  year={2016},
  doi={10.1103/PhysRevE.93.032314},
}

@inproceedings{friedrich2015cliques,
  title={Cliques in hyperbolic random graphs},
  author={Friedrich, Tobias and Krohmer, Anton},
  booktitle={Conference on Computer Communications (INFOCOM)},
  pages={1544--1552},
  year={2015},
  doi={10.1109/INFOCOM.2015.7218533},
}

@article{friedrich2018diameter,
  title={On the diameter of hyperbolic random graphs},
  author={Friedrich, Tobias and Krohmer, Anton},
  journal={SIAM Journal on Discrete Mathematics},
  volume={32},
  number={2},
  pages={1314--1334},
  year={2018},
  doi={10.1137/17M1123961}
}

@book{book:694330,
   title =     {Spectral Graph Theory},
   author =    {Fan R. K. Chung},
   publisher = {American Mathematical Society},
   isbn =      {978-0-8218-0315-8},
   year =      {1997},
   series =    {},
   edition =   {},
   volume =    {},
}

@article{coja2007laplacian,
  title={On the Laplacian eigenvalues of $G_{n, p}$},
  author={Coja-Oghlan, Amin},
  journal={Combinatorics, Probability and Computing},
  volume={16},
  number={6},
  pages={923--946},
  year={2007},
  doi={10.1017/S0963548307008693},
}

@article{barabasi1999emergence,
  title={Emergence of scaling in random networks},
  author={Albert-László Barabasi and Réka Albert},
  journal={Science},
  volume={286},
  number={5439},
  pages={509--512},
  year={1999},
  doi={10.1126/science.286.5439.509},
}

@article{boguna2010sustaining,
  title={Sustaining the internet with hyperbolic mapping},
  author={Bogun{\'a}, Mari{\'a}n and Papadopoulos, Fragkiskos and Krioukov, Dmitri},
  journal={Nature Communications},
  volume={1},
  number={1},
  pages={1--8},
  year={2010},
  doi={10.1038/ncomms1063},
}

@article{verbeek2014metric,
  title={Metric embedding, hyperbolic space, and social networks},
  author={Verbeek, Kevin and Suri, Subhash},
  journal={Computational Geometry},
  volume={59},
  pages={1--12},
  year={2016},
  doi={10.1016/j.comgeo.2016.08.003},
}

@article{bode2015largest,
  title={On the largest component of a hyperbolic model of complex networks},
  author={Bode, Michel and Fountoulakis, Nikolaos and M{\"u}ller, Tobias},
  journal={The Electronic Journal of Combinatorics},
  volume={22},
  number={3},
  pages={1--52},
  year={2015},
  doi={10.37236/4958},
}

@inproceedings{gugelmann2012random,
  title={Random hyperbolic graphs: degree sequence and clustering},
  author={Gugelmann, Luca and Panagiotou, Konstantinos and Peter, Ueli},
  booktitle={International Colloquium on Automata, Languages, and Programming (ICALP)},
  pages={573--585},
  year={2012},
  doi={10.1007/978-3-642-31585-5_51},
}

@article{muller2019diameter,
  title={The diameter of KPKVB random graphs},
  author={M{\"u}ller, Tobias and Staps, Merlijn},
  journal={Advances in Applied Probability},
  volume={51},
  number={2},
  pages={358--377},
  year={2019},
  doi={10.1017/apr.2019.23},
}

@article{erdHos1959random,
  title={On random graphs I},
  author={Erd{\H{o}}s, Paul and R{\'e}nyi, Alfr{\'e}d},
  journal={Publicationes Mathematicae},
  volume={6},
  number={1},
  pages={290--297},
  year={1959}
}

@article{ChungL03ChungLuGraphs,
    author    = {Fan R. K. Chung and Linyuan Lu},
    title     = {The average distance in a random graph with given expected degrees},
    journal   = {Internet Mathematics},
    volume    = {1},
    number    = {1},
    pages     = {91--113},
    year      = {2003},
    doi       = {10.1080/15427951.2004.10129081},
}

@article{lyapunov1992general,
  title={The general problem of the stability of motion},
  author={Lyapunov, Aleksandr Mikhailovich},
  journal={International journal of control},
  volume={55},
  number={3},
  pages={531--534},
  year={1992},
  publisher={Taylor \& Francis}
}

\end{document}